\documentclass[a4paper,12pt,reqno]{amsart}
\usepackage{amsfonts,amssymb,hyperref,amsthm,enumerate,
	color,stmaryrd,pgf,tikz,comment,multirow}
\usepackage{amsmath}
\usepackage{tikz-cd}
\usepackage[left=2.6cm,right=2.6cm,top=3cm,bottom=3cm,bindingoffset=0cm]{geometry}
\usepackage{booktabs,caption}
\usepackage[flushleft]{threeparttable}

\newtheorem{theorem}{Theorem}[section]
\newtheorem{lemma}[theorem]{Lemma}
\newtheorem{proposition}[theorem]{Proposition}

\theoremstyle{definition}
\newtheorem{remark}{Remark}

\newtheorem*{case}{Case}
\newtheorem*{subcase}{Subcase}

\DeclareMathOperator{\Fix}{Fix}
\DeclareMathOperator{\Ker}{Ker}
\DeclareMathOperator{\Core}{Core}
\DeclareMathOperator{\Aut}{Aut}

\DeclareMathOperator{\Skew}{Skew}
\DeclareMathOperator{\Smooth}{Smooth}

\DeclareMathOperator{\SD}{SD}
\DeclareMathOperator{\CM}{CM}

\DeclareMathOperator{\ch}{\,char\,}
\DeclareMathOperator{\av}{av}

\title{Classification of regular Cayley maps of skew-type three on semidihedral groups}
\author[K. Hu\and T. Qiu]
{ Kan Hu$^{*}$\and Tao Qiu}
\address{K. Hu,
\newline\indent
Department of Mathematics, Zhejiang Ocean University, Zhoushan, Zhejiang 316022, P.R. China
}
\email{hukan@zjou.edu.cn}

\address{T. Qiu,
\newline\indent
Department of Mathematics, Zhejiang Ocean University, Zhoushan, Zhejiang 316022, P.R. China
}
\email{2469660019@qq.com}

\thanks{This work was supported by National Natural
Science Foundation of China (12471332).
\newline\indent
$^{\ast}$ Corresponding author e-mail: hukan@zjou.edu.cn
}
\begin{document}
\maketitle

\begin{abstract}
It is well known that every regular Cayley map $M = \CM(G,X,p)$ on a finite group $G$ with 
respect to an inverse-closed generating set $X$ of $G$ and a specified cyclic permutation 
$p$ on $X$ corresponds to a skew morphism $\varphi$ on $G$ such that the restriction of 
$\varphi$ to $X$ is $p$. The skew-type of the map $M$ is defined as the index $[G:\Ker \varphi]$, 
which equals the number of distinct values in $\mathbb{Z}_{|\varphi|}$ taken by the associated 
power function $\pi$ of the skew morphism $\varphi$. In this paper, we develop a covering theory of  
skew morphisms and as an application we provide a classification of regular Cayley maps of skew-type three on the semidihedral groups.
\end{abstract}

\section{Introduction}
Throughout the paper, all groups and graphs are assumed finite unless stated otherwise. A \textit{map} 
$M$ is an embedding $i:\Gamma\hookrightarrow\mathbb{S}$ of a connected graph $\Gamma$ 
into a closed surface $\mathbb{S}$ such that each component of $\mathbb{S}\setminus i(\Gamma)$ 
is homeomorphic to an open disc. The map $M$ is \textit{orientable} if its underlying surface $\mathbb{S}$ 
is orientable; otherwise it is called \textit{non-orientable}. In this paper, all maps are assumed to be orientable.
 In this case, an \textit{automorphism} of the orientable map $M$ is a permutation of the arcs of the embedded 
 graph $\Gamma$ that respects the graph structure and extends to an orientation-preserving 
 self-homeomorphism of the carrier surface $\mathbb{S}$. It is well known that the automorphism 
 group $\Aut(M)$ of $M$ acts semiregularly on the arcs. If this action is transitive (hence regular), 
 then the map itself is called a \textit{regular} map.

An important problem in topological graph theory is the construction and classification of regular 
maps, usually under certain reasonable conditions on the embedded graphs, the underlying surfaces, 
or the automorphism groups. Regular Cayley maps on a given group, due to their inherent highly 
symmetric structure, have attracted much attention~\cite{RSJTW2005}.

A \textit{Cayley map} $M=\CM(G,X,p)$ on a group $G$ with respect to an inverse-closed generating 
subset $X\subseteq G\setminus\{1_G\}$ and a cyclic permutation $p$ on $X$ is the embedding of 
the Cayley graph $\Gamma:=\mathrm{Cay}(G,X)$ into an orientable closed surface $\mathbb{S}$ 
such that its local rotation $\rho$ at each vertex determined by the embedding is consistent with 
the permutation $p$, namely,
\[
\rho(g,gx)=(g,gp(x))\quad\text{for all }g\in G\text{ and }x\in X.
\]
The left translation $L_a:g\mapsto ag$ ($a\in G$) induces a subgroup $G_L$ of $\Aut(M)$ acting 
regularly on the vertices of $M$. Regarding $G_L$ as a permutation group on $G$, it follows from 
the Frattini argument that $\Aut(M)=G_L C$ is an exact product of $G_L$ and the (cyclic) stabilizer 
$C:=(\Aut M)_{1_G}$ of the vertex $1_G$. A chosen generator $c$ of $C$ determines a \textit{skew morphism} 
$\varphi$ on $G$ via the identity $cg=\varphi(g)c^{\pi(g)}$ satisfying the following defining identities: 
$\varphi(1_G)=1$, and $\varphi(gh)=\varphi(g)\varphi^{\pi(g)}(h)$ for all $g,h\in G$, where 
$\pi:G\to\mathbb{Z}_{|\varphi|}$ is the associated \textit{power function}. A seminal result by Jajcay 
and \v{S}ir\'{a}\v{n} shows that the Cayley map $M$ is regular iff the restriction of 
$\varphi$ to $X$ is $p$, i.e., $\varphi|_X=p$~\cite{JS2002}. Thus, for a given group $G$, the 
determination of regular Cayley maps $M$ on $G$ is equivalent to the determination of 
\textit{Cayley skew morphisms} $\varphi$ on $G$, that is, skew morphisms on $G$ with 
an inverse-closed generating orbit.

As noted, the automorphism group $\Aut(M)$ of a regular Cayley map $M$ on $G$ can be 
viewed as an exact product of $G$ and a cyclic \textit{core-free} subgroup $C$, which is referred 
to as the \textit{skew-product group} of $G$. In general, an exact product $A=GC$ of a subgroup $G$ and 
a cyclic (not necessarily core-free) subgroup $C$ is also called a \textit{cyclic complementary extension} of $G$. In this case,
a chosen generator $c$ of $C$ induces a skew morphism $\varphi$ on $G$ and an 
\textit{extended power function} of $\varphi$; conversely, given a skew morphism $\varphi$ of 
$G$ and an associated extended power function, there is a canonical way to construct a cyclic 
complementary extension of $G$~\cite{HJ2026}. It is therefore natural to investigate cyclic 
complementary extensions of a given group; see~\cite{CJT2016,DH2019,HJ2026,HKK2022,HKK2025,KN2011,KN2017,ML2024,Yu2026,ZD2016} for results
in this broader direction.

Among other things, the kernel $\Ker \varphi$ of a skew morphism $\varphi$ on $G$, a subgroup 
consisting of elements $g\in G$ with $\pi(g)\equiv1\pmod{|\varphi|}$, plays an important role in 
the study of regular Cayley maps and skew morphisms. The index $[G:\Ker \varphi]$, 
called the \textit{skew-type} of $M$ or $\varphi$, is equal to the number of distinct values taken by 
the power function $\pi$ in $\mathbb{Z}_{|\varphi|}$. This is an important arithmetic invariant of 
skew morphisms, and it is clear that the skew-type is smaller the closer the skew morphism is to being an automorphism.

 A central problem in this area is the construction and classification of regular Cayley maps for a given family of groups 
$G$. For cyclic groups $\mathrm{C}_n$, a complete solution was provided by Conder and Tucker~\cite{CT2014}. 
Considerable effort was subsequently devoted to the dihedral groups 
$\mathrm{D}_{2n}$~\cite{KK2016,KK2017,KKF2006,WF2005,Zhang2015a,Zhang2015b,ZD2016}, until the problem was eventually resolved by Kov\'acs and Kwon~\cite{KK2021}. More recently, the problem has also been settled for the elementary abelian 
$p$-groups $\mathbb{Z}_p^n$ by Du, Luo and Yu~\cite{DLY2025}.

Only partial results have been obtained for the classification of regular Cayley maps on other families of groups, 
including abelian groups~\cite{CJT2007-1,CJT2007-2,DYL2023}, certain metacyclic groups 
such as the generalized quaternion groups $\mathrm{Q}_{4n}$~\cite{WF2005,KO2008}, and the 
semidihedral groups $\SD_{8n}$~\cite{Oh2009}. In most cases the authors focused on balanced 
or $t$-balanced regular Cayley maps whose skew-types are either $1$ or $2$.

In this paper, we focus on extending the classification of regular Cayley maps to the semidihedral group $\SD_{8n}$, given by the presentation
\begin{equation}\label{Pre}
\SD_{8n}=\langle a,b\mid a^{4n}=b^2=1,\; bab^{-1}=a^{2n-1}\rangle,\quad n\geq 2.
\end{equation}
Although Yu~\cite{Yu2026} recently characterized the cyclic complementary extensions of $\SD_{8n}$, the regular Cayley maps themselves have yet to be fully classified, apart from the $t$-balanced examples previously obtained by Oh~\cite{Oh2009}. To complete this classification, the natural next step is to consider regular Cayley maps of skew-type $3$. To this end, we adapt the methods used in \cite{KMM2013, Zhang2015a} for the dihedral group $\mathrm{D}_{2n}$ to this setting. More precisely, we develop a comprehensive theory of coverings of skew morphisms and establish the following classification theorem for regular Cayley maps of skew-type $3$ on the semidihedral group.

\begin{theorem}\label{Main}
Let $M=\CM(\SD_{8n},X,p)$ be a regular Cayley map of skew-type $3$ on the semidihedral group 
$\SD_{8n}$. Then $n$ is divisible by $3$, and, up to isomorphism, $M$ belongs to one of the two families:
\begin{enumerate}[\rm(a)]
\item For $n=3$, $M\cong \CM(\SD_{24}, X, p)$, where
\[
X=\{a, a^3, a^5, a^7, a^9, a^{11}, ab, a^3b, a^5b, a^7b, a^9b, a^{11}b\}
\]
and
\[
p = (a^3, ab, a^5, a^3b, a^7, a^5b, a^9, a^7b, a^{11}, a^9b, a, a^{11}b).
\]
\item For all $n$ divisible by $3$, $M\cong M(n;t):=\CM(\SD_{8n},X,p)$, where $t\in\mathbb{Z}_{4n}^*$ 
is an integer of odd multiplicative order $k:=o_{4n}(t)$ in $\mathbb{Z}_{4n}$ such that 
$t \equiv 1 \pmod{6}$, $X=\{x_i\}_{0\leq i<4k}$ and $p=(x_i)_{0\leq i<4k}$ with each $x_i$ given by
\begin{equation}\label{Perm0}
x_i=
\begin{cases}
a^{t^i}, & \text{if } i\equiv0\pmod{4},\\[2pt]
a^{-t^{i}}, & \text{if } i\equiv1\pmod{4},\\[2pt]
a^{t^i+3}b, & \text{if } i\equiv2\pmod{4},\\[2pt]
a^{-t^i+3+2n}b, & \text{if } i\equiv3\pmod{4}.
\end{cases}
\end{equation}
\end{enumerate}
Moreover, maps from distinct families, or from the same family with distinct parameters, are mutually non-isomorphic.
\end{theorem}

\begin{remark}
Let $\varphi$ be the corresponding skew morphism of the maps in Theorem~\ref{Main}.   
The single map in (a) is exceptional in the sense that $\Core\varphi=\langle a^6\rangle$, contrary to  
  the infinite family of maps in (b)  where $\Core\varphi=\langle a^3\rangle$; see Lemma~\ref{lemma Core}.
\end{remark}

The paper is organized as follows. In the next preliminary section, we collect necessary results
to be used later, with emphasis on developing a theory on
coverings of skew morphisms and regular Cayley maps. In Section~3, we 
 present characterization lemmas on the subgroups of $\SD_{8n}$ and the core $\Core\varphi$ of a 
 skew morphism of skew-type $3$ on $\SD_{8n}$. In Section~4, we verify that the maps stated in 
 Theorem~\ref{Main} are indeed regular Cayley maps of skew-type $3$. Finally, in Section~5, we 
 present a proof of Theorem~\ref{Main}.

\section{Preliminaries}
In this section, we collect preliminary results on skew morphisms and regular Cayley maps for 
further reference, with emphasis on developing a comprehensive theory on coverings of skew morphisms.

\subsection{Skew morphisms}
Let $G$ be a finite group. A \textit{skew morphism} on $G$ is a permutation $\varphi$ on $G$ fixing 
the identity element for which there exists an integer-valued function $\pi:G\to\mathbb{Z}_{|\varphi|}$, 
called the \textit{power function} associated with $\varphi$, such that
\[
\varphi(gh)=\varphi(g)\varphi^{\pi(g)}(h)\quad\text{for all }g,h\in G.
\]
Suppose that $\varphi$ is a skew morphism on a group $G$ with power function $\pi:G\to\mathbb{Z}_{m}$, 
where $m:=|\varphi|$. If $\delta:G\to H$ is an isomorphism, then $\psi:=\delta\varphi\delta^{-1}$ is a skew 
morphism on $H$ with power function $\pi\delta^{-1}$.
More precisely, a skew morphism $\varphi$ on $G$ is said to be \textit{equivalent} to a skew morphism $\psi$ on $H$ if
there is an isomorphism $\delta:G\to H$ such that $\psi\delta=\delta\varphi$.

Now suppose that $\varphi$ and $\psi$ are skew morphisms on $G$, and $\Delta\leq\Aut(G)$. 
If there exists $\delta\in \Delta$ with $\delta\varphi=\psi\delta$, then $\varphi$ is said to be 
\textit{$\Delta$-equivalent} to $\psi$. It is clear that if $\Xi\leq \Delta\leq\Aut(G)$, 
and $\varphi$ and $\psi$ are $\Xi$-equivalent, then they are also $\Delta$-equivalent (but the
converse may not be true). In particular, if $\varphi$ and $\psi$ are $\Aut(G)$-equivalent,
then we will simply say that they are \textit{equivalent}.
Thus, the automorphism group $\Aut(G)$ acts by conjugation on the set $\Skew(G)$ of
skew morphisms on $G$.

A skew morphism $\varphi$ on $G$ is called a \textit{covering} of a skew morphism
$\psi$ on $H$, denoted by $\nu:\varphi\to\psi$, if there is an  epimorphism $\nu:G\to H$ from
 $G$ onto $H$ such that  $\nu\varphi=\psi\nu$. In this case, we also say 
 that $\varphi$ is a \textit{lifting} of $\psi$, and $\psi$ is a \textit{projection} of $\varphi$. 
Note that $\Ker \nu$ is a $\varphi$-invariant normal subgroup of $G$, and the $\psi$-invariant (normal) 
subgroups $S$ of $H$ are in bijective correspondence with the $\varphi$-invariant (normal) subgroups 
$\nu^{-1}(S)$ of $G$ containing $\Ker \nu$~\cite{WHYZ2019}. 

Conversely,  if $N$ is a $\varphi$-invariant normal subgroup of $G$ for
some skew morphism $\varphi$ on $G$ with power function $\pi$, 
then the mapping $\overline{\varphi}:\overline{G}\to\overline{G}$ 
defined by $\overline{\varphi}( g N)=\varphi(g)N$ ($g\in G$) is a skew morphism on the quotient group 
$\overline{G}:=G/N$. This is called the \textit{quotient skew morphism} induced by $N$.
Note that if $\nu:G\to\overline{G}$ is the natural epimorphism, then 
$\nu\varphi(g)=\varphi(g)N=\overline{\varphi}\nu(g)$ for all $g\in G$,
so $\nu$ is a covering map from $\varphi$ to $\overline{\varphi}$. It follows that
 $\nu=\nu\varphi^m=\overline{\varphi}^m\nu$,  where $m:=|\varphi|$, so $\overline{\varphi}^m=\mathrm{id}_{\overline{G}}$. 
 Therefore, the order
  of $\overline{\varphi}$ divides the order of $\varphi$, and it is easy to see that 
the  power function  $\bar\pi:\overline{G}\to\mathbb{Z}_{\overline{m}}$  of $\overline{\varphi}$ 
is determined by $\bar\pi(g N)\equiv\pi(g)\pmod{\overline{m}}$, where $\overline{m}:=|\overline{\varphi}|$.

Suppose that $\varphi$ and $\varphi'$ are skew morphisms on $G$, $\psi$ 
and $\psi'$ are skew
morphisms on $H$, and  $\nu: \varphi\to\psi$ and $\nu': \varphi'\to\psi'$ are coverings. 
The two coverings $\varphi$ and $\varphi'$ are said to be \textit{equivalent} if there exist $\delta\in\Aut(G)$
and $\tau\in\Aut(H)$ such that $\nu'\delta\varphi\delta^{-1}=\tau\psi\tau^{-1}\nu':$
\[
\begin{tikzcd} \varphi \ar[r, "\delta"] \ar[d, "\nu"'] & \varphi' \ar[d, "\nu'"] \\ \psi \ar[r, "\tau"] & \psi' \end{tikzcd}
\]
where the horizontal  arrows $\varphi\stackrel{\delta}\to\varphi'$ and $\psi\stackrel{\tau}\to\psi'$, and the vertical arrows 
$\varphi\stackrel{\nu}\to\psi$ and $\varphi'\stackrel{\nu'}\to\psi'$ are understood as the following equations:
\[
\delta\varphi=\varphi'\delta,\quad\tau\psi=\psi'\tau,\quad \nu\varphi=\psi\nu \quad\text{and}\quad\nu'\varphi'=\psi'\nu'.
\]
In particular, if $\psi=\psi'$ so that $\varphi$ and $\varphi'$ are coverings of a common skew morphism
$\psi$ on $H$, then the above condition for the equivalence of the coverings $\varphi$ and $\varphi'$  reduces
to the existence of an automorphism $\delta$ of $G$ such that $\varphi'=\delta\varphi\delta^{-1}$.

\begin{proposition}
Let $\varphi$ and $\psi$ be equivalent skew morphisms on $G$.
If $N$ is a normal subgroup of $G$ which is both $\varphi$-invariant and $\psi$-invariant,  
then the quotient skew morphisms $\overline{\varphi}$ and $\overline{\psi}$ on $G/N$  induced by $N$ 
 are equivalent. In particular, if $N$ is a characteristic subgroup of $G$, then for every
 skew morphism $\varphi$ of $G$ with $\varphi(N)=N$,  the
 equivalence class containing $\varphi$ is projected via the covering $\nu:G\to G/N$ 
 to the equivalence class containing $\overline{\varphi}$.

\end{proposition}
\begin{proof}
Since $\varphi$ and $\psi$ are equivalent, there exists $\delta\in\Aut(G)$ with $\delta\varphi=\psi\delta$.
Since $N\lhd G$ is both $\varphi$-invariant and $\psi$-invariant, we have $\delta(N)=N$. It is easy to
prove that the mapping $\delta^*:gN\to\delta(g)N$ is an automorphism of $G/N$ with $\delta^*\overline{\varphi}=\overline{\psi}\delta^*$
Thus, $\overline{\varphi}$ and $\overline{\psi}$ are equivalent. Moreover, if $N\ch G$ and $\varphi(N)=N$, then for 
every skew morphism $\psi$ equivalent to $\varphi$, we have $\psi(N)=N$, so $\overline{\psi}$ is equivalent to $\overline{\varphi}$,
as required.
\end{proof}

The \textit{kernel} of a skew morphism $\varphi$ on $G$ with power function $\pi:G\to\mathbb{Z}_{|\varphi|}$
is the subgroup $\Ker\varphi$  of $G$ defined by
\[
\Ker \varphi = \{ g \in G \mid \pi(g) \equiv 1 \pmod{|\varphi|} \}.
\]
The index $[G:\Ker \varphi]$ is called the \textit{skew-type} of $\varphi$~\cite{Zhang2015a}. 
It is evident that a skew morphism  is an automorphism iff it has skew-type $1$. Thus skew 
morphisms of small skew-type are closer to being group automorphisms. Moreover, 
if $\varphi$ and $\psi$ are equivalent skew morphisms, then they have the same skew-type.
Thus, the skew-type is an \textit{invariant}  for the category of skew morphisms, 
that is, a property that remains unchanged under equivalence of skew morphisms.

The following results are fundamental.
\begin{proposition}[\cite{JS2002}]\label{Formula}
Let $\varphi$ be a skew morphism of a finite group $G$ and let $\pi:G\to\mathbb{Z}_m$ be its power function, where $m:=|\varphi|$. Then the following hold for all $g,h\in G$:
\begin{enumerate}[\rm(a)]
\item $\pi(g) \equiv \pi(h)\pmod{m}$ iff $gh^{-1}\in\Ker \varphi$.
\item $\pi(gh) \equiv \sum_{i=0}^{\pi(g)-1} \pi(\varphi^i(g)) \pmod{m}$.
\item $\varphi^k(gh)= \varphi^k(g) \varphi^{\sum_{i=0}^{k-1} \pi(\varphi^i(g))}(h)$ for any integer $k$.
\end{enumerate}
\end{proposition}

Another important arithmetic invariant of skew morphisms is the \textit{period} of the associated power 
function discovered by Bachrat\'y and Jajcay~\cite{BJ2016}. It is defined, for any skew morphism $\varphi$ 
on $G$ with power function $\pi:G\to\mathbb{Z}_{|\varphi|}$, as the smallest positive integer $\ell$ such that
\[
\pi(\varphi^\ell(g))\equiv\pi(g)\pmod{|\varphi|}\quad\text{for all }g\in G.
\]
It is known that if $\varphi$ has period $\ell$, then $\psi:=\varphi^\ell$ is a \textit{smooth} (or \textit{coset-preserving}) 
skew morphism on $G$ in the sense that its power function is constant on each orbit of $\psi$~\cite{BJ2016,WHYZ2019}. 
The power function of $\psi$ is given by $\av:G\to\mathbb{Z}_{m/\ell}$, where
\[
\av(g)=\frac{1}{\ell}\sum_{i=1}^\ell\pi(\varphi^{i-1}(g))\pmod{m/\ell},\quad g\in G.
\]
This function is also known as the \textit{average function} of $\pi$ and is always a homomorphism from $G$ 
into the multiplicative group $\mathbb{Z}^*_{m/\ell}$~\cite{HJ2026,HR2022}.

There is another arithmetic invariant of skew morphisms called the \textit{auto-index}
of a skew morphism $\varphi$ on $G$, which is defined as the 
smallest positive integer $h$ such that $\varphi^h$ is an automorphism of $G$. It is clear that the period $\ell$ 
of a skew morphism $\varphi$ divides its auto-index $h$, and $h$ divides $m:=|\varphi|$~\cite{HKZ2021}.

The following results on skew morphisms are well known.
\begin{proposition}[\cite{HKZ2021,WHYZ2019}]\label{Parameter}
Let $\varphi$ be a skew morphism of a finite group $G$ with power function $\pi:G\to\mathbb{Z}_m$, where 
$m:=|\varphi|$. Then for any integer $k$:
\begin{enumerate}[\rm(a)]
\item $\psi:=\varphi^k$ is a skew morphism on $G$ iff there exists a function 
$\pi_\psi:G\to\mathbb{Z}_{m/\gcd(k,m)}$ such that $k\pi_\psi(g)=\sum_{i=0}^{k-1}\pi(\varphi^i(g))\pmod{m}$ 
for all $g\in G$; in that case $\pi_\psi$ is the power function of $\varphi^k$.
\item $\varphi^k$ is an automorphism of $G$ iff $k$ is divisible by the auto-index of $\varphi$; 
in that case $\sum_{i=0}^{k-1}\pi(\varphi^i(g))\equiv k\pmod{|\varphi|}$ for all $g\in G$.
\item $\psi:=\varphi^k$ is a smooth skew morphism on $G$ iff $k$ is divisible by the period of $\pi$; 
in that case its power function $\pi_\psi$ is given by $\pi_\psi(g)=\frac{1}{k}\sum_{i=0}^{k-1}\pi(\varphi^i(g))$ for all $g\in G$.
\end{enumerate}
\end{proposition}

An important $\varphi$-invariant normal subgroup of $G$ is the \textit{core} of $\varphi$ defined by
\[
\Core\varphi =\bigcap_{i=1}^{|\varphi|}\varphi^i(\Ker \varphi).
\]
This is the largest $\varphi$-invariant normal subgroup of $G$ contained in $\Ker \varphi$~\cite{Zhang2015a}. 
The quotient skew morphism $\overline{\varphi}$ on $\overline{G}:=G/\Core\varphi$ induced by $\Core\varphi$ is closely 
related to the period of $\varphi$.

\begin{proposition}\cite[Theorem 4.5]{WHYZ2019}\label{Period}
Let $\varphi$ be a skew morphism on $G$
of period $\ell$, and let $\overline{\varphi}$ be the quotient skew morphism on 
$\overline{G}:=G/\Core\varphi$ induced by $\Core\varphi$. Then:
\begin{enumerate}[\rm(a)]
\item $\ell$ is equal to the order of $\overline{\varphi}$.
\item The inverse image $\Smooth\varphi:=\nu^{-1}(\Fix\overline{\varphi})$ of the fixed point set 
$\Fix\overline{\varphi}$ of $\overline{\varphi}$ is the largest $\varphi$-invariant subgroup of $G$ 
with the property that $\pi$ is constant on the orbit $O_g$ for each $g\in\Smooth\varphi$.
\item $\varphi$ is smooth iff $\overline{\varphi}=\mathrm{id}_{\overline{G}}$ is the identity automorphism of $\overline{G}$.
\end{enumerate}
\end{proposition}

\begin{proposition}\label{quotient}
Let $\varphi$ be a skew morphism on $G$, and let $\overline{\varphi}$ be the quotient 
skew morphism on $\overline{G}:=G/N$ induced by a  $\varphi$-invariant normal subgroup $N$ of $G$.
Then the skew-type, period, and auto-index of $\overline{\varphi}$ divide 
the skew-type, period, and auto-index of $\varphi$, respectively.
\end{proposition}
\begin{proof}
If $g\in\Ker\varphi$, then $\pi(g)\equiv1\pmod{|\varphi|}$. Since $|\overline{\varphi}|$ divides $|\varphi|$, 
we have $\bar\pi(g)\equiv\pi(g)\equiv1\pmod{|\overline{\varphi}|}$, 
and hence $\bar g\in\Ker\overline{\varphi}$. Therefore, $\overline{N\Ker\varphi}\leq\Ker\overline{\varphi}$. Now
\begin{align*}
[G:\Ker\varphi]&=[G:N\Ker\varphi][N\Ker\varphi:\Ker\varphi]\\
&=[\overline{G}:\overline{N\Ker\varphi}][N\Ker\varphi:\Ker\varphi]\\
&=[\overline{G}:\Ker\overline{\varphi}][\Ker\overline{\varphi}:\overline{N\Ker\varphi}][N\Ker\varphi:\Ker\varphi],
\end{align*}
so $[\overline{G}:\Ker\overline{\varphi}]$ divides $[G:\Ker\varphi]$. Moreover, let $\ell$ and $\bar\ell$ be the periods of $\varphi$ 
and $\overline{\varphi}$, then for all $g\in G$, we have $\pi(\varphi^\ell(g))\equiv\pi(g)\pmod{|\varphi|}$; since $|\overline{\varphi}|$ divides
$|\varphi|$, we also have $\pi(\varphi^\ell(g))\equiv\pi(g)\pmod{|\overline{\varphi}|}$, so the minimality of $\bar\ell$ implies that
 $\bar\ell$ divides $\ell$. 
Finally, let $h$ and $\bar h$ be the auto-indices of $\varphi$ and $\overline{\varphi}$, then $\varphi^h$ is an automorphism of $G$, so
$\overline{\varphi}^h$ is also an automorphism of $G/N$. Thus, $\bar h$ divides $h$, as required.
\end{proof}

\subsection{Regular Cayley maps}
The original motivation for introducing skew morphisms was to characterize regular Cayley maps, as
shown below:
\begin{proposition}[\cite{JS2002}]
A Cayley map $M=\CM(G,X,p)$ on a finite group $G$ is regular iff there exists a skew 
morphism $\varphi$ of $G$ with $\varphi|_X=p$.
\end{proposition}
Since $X$ is closed under taking inverses, one may define the \textit{distribution-of-inverses} function 
$\chi: X\to\mathbb{N}$~\cite{JS2002}, where $\chi(x)$ is the smallest nonnegative integer such that
\[
p^{\chi(x)}(x)=x^{-1}\quad\text{for each }x\in X.
\]
The relationship between $\chi$ and the power function $\pi$ of $\varphi$ is given by
\begin{align}\label{Dist}
\pi(x) &\equiv \chi(\varphi(x)) - \chi(x) + 1 \pmod{|X|},\qquad x\in X.
\end{align}
If the skew morphism $\varphi$ has period $\ell$, then for all $x\in X$,
\[
\chi(\varphi^{\ell+1}(x)) - \chi(\varphi^{\ell}(x)) + 1
\equiv \pi(\varphi^\ell(x)) \equiv \pi(x)
\equiv \chi(\varphi(x)) - \chi(x) + 1 \pmod{|X|},
\]
or equivalently,
\[
\chi(\varphi^{\ell+1}(x)) -\chi(\varphi(x)) \equiv \chi(\varphi^{\ell}(x)) -\chi(x) \pmod{|X|}.
\]
Thus the difference $\chi(\varphi^{\ell}(x)) -\chi(x)$ is constant modulo $|X|$ for all $x\in X$; it will be called 
the \textit{common difference} of $\chi$. It has the following properties.

\begin{lemma}\label{CD}
Let $M=\CM(G,X,p)$ be a regular Cayley map and $\varphi$ the corresponding skew morphism with power 
function $\pi$ and distribution-of-inverses function $\chi$ on $X$. Let $\ell$ be the period of $\pi$ and $d$ 
the common difference of $\chi$. Then:
\begin{enumerate}[\rm(a)]
\item $d\equiv \ell (\av(x)-1)\pmod{|X|}$ for all $x\in X$; in particular, $\ell$ divides $d$.
\item If $\varphi^\ell$ is an automorphism of $G$, then $d=0$.
\item If there exists an involution $x\in X$ such that $\varphi^\ell(x)$ is also an involution, then $d=0$.
\end{enumerate}
\end{lemma}
\begin{proof}
Write $m:=|X|$. For any $x\in X$,
\begin{align*}
\ell \av(x)&=\sum_{i=1}^\ell\pi(\varphi^{i-1}(x))
\equiv\sum_{i=1}^\ell\bigl(\chi(\varphi^{i}(x))-\chi(\varphi^{i-1}(x))+1\bigr)\\
&=\chi(\varphi^{\ell}(x))-\chi(x)+\ell\equiv d+\ell\pmod{m}.
\end{align*}
Hence $\ell\mid d$ because $\ell\mid m$. Now $\av:G\to\mathbb{Z}_{m/\ell}$ is the power function of 
$\varphi^\ell$, so if $\varphi^\ell$ is an automorphism then $\av(x)=1$ for all $x\in X$ and $d=0$. Finally, 
if there exists an involution $x\in X$ such that $\varphi^\ell(x)$ is also an involution, then 
$\chi(x)=0=\chi(\varphi^\ell(x))$, so $d=\chi(\varphi^\ell(x))-\chi(x)=0$.
\end{proof}

Now we turn to coverings between regular Cayley maps. The following is a generalization of Lemma 2.4 in~\cite{KKF2006}.
\begin{proposition}
Let $M=\CM(G,X,p)$ and $M'=\CM(G',X',p')$ be regular Cayley maps, and let $\varphi$ and $\varphi'$ 
be the corresponding skew morphisms on $G$ and $G'$, respectively. Then:
\begin{enumerate}[\rm(a)]
\item Every covering $\nu:G\to G'$ from $\varphi$ to $\varphi'$ with $\nu(X)=X'$ induces a regular 
covering from $M$ to $M'$. Conversely, every regular covering from $M$ to $M'$ restricts to a covering $\nu:G\to G'$ from 
$\varphi$ to $\varphi'$ with $\nu(X)=X'$.
\item Two regular Cayley maps $\CM(G,X,p)$ and $\CM(G,Y,q)$ on $G$ are isomorphic iff
there is an automorphism $\delta$ of $G$ such that $\delta(X)=Y$ and $q\delta(x)=\delta p(x)$ for all $x\in X$.
\end{enumerate}
\end{proposition}
\begin{proof}
(a) Suppose that 
\[
M=\CM(G,X,p)=(\delta(M);\rho,\lambda)\quad\text{and}\quad M'=\CM(G',X',p')=(\delta(M');\rho',\lambda'),
\] 
where $\rho,\lambda,\varphi$ (resp. $\rho',\lambda',\varphi'$) are the local rotations, dart-reversing involutions, 
and skew morphisms corresponding to $M$ (resp. $M'$). If $\nu:G\to G'$ is a covering from $\varphi$ to $\varphi'$ 
with $\nu(X)=X'$, define $\mu: M\to M'$ by $\mu(g,gx)=(\nu(g),\nu(g)\nu(x))$. Then for 
any $(g,gx)\in\Omega(M)$, we have
\begin{align*}
\mu\rho(g,gx)&=\mu(g,gp(x))=(\nu(g),\nu(g)\nu p(x))=(\nu(g),\nu(g)p'\nu(x))=\rho'\mu(g,gx),\\
\mu\lambda(g,gx)&=\mu(gx,g)=(\nu(g)\nu(x),\nu(g))=\lambda'(\nu(g),\nu(g)\nu(x))=\lambda'\mu(g,gx).
\end{align*}
Thus, $\mu$ is a regular covering from $M$ to $M'$.

Conversely, suppose that $\mu: M\to M'$ is a regular covering from $M$ to $M'$. For each 
$g\in G$ and $x\in X$, there exist unique $g'\in G'$ and $x'\in X'$ such that $\mu(g,gx)=(g',g'x')$. 
By composing an automorphism of $M'$ if necessary, we may assume $\mu(1,x)=(1',x')$ 
for some $x\in X$, where $1'$ is the identity element of $G'$.
Define $\nu:G\to G'$ by $\nu(g)=g'$. From $\mu\rho=\rho'\mu$ and $\mu\lambda=\lambda'\mu$, 
we deduce that $\nu\rho^i(x)={\rho'}^i\nu(x)$ and $\nu(x^{-1})=\nu(x)^{-1}$ for all $i\in\mathbb{Z}_{|X|}$.
Now it is routine to use induction on the length of words in terms of elements of $X$ to show
 that  $\nu$ is an epimorphism with $\nu\varphi=\varphi'\nu$. Therefore, $\nu$ is a covering
 from $\varphi$ to $\varphi'$, as required.
 
 (b) This follows immediately from (a); see also~\cite[Lemma 2.4]{KKF2006}.
\end{proof}

\begin{proposition}\label{Lift}
Let $\varphi$ and $\varphi'$ be Cayley skew morphisms on $G$ and $G'$, with specified Cayley orbits $X$ and $X'$, 
respectively. Suppose $\nu:G\to G'$ is a covering map from $\varphi$ to $\varphi'$ with $\nu(X)=X'$. If $S\subseteq X'$ 
is inverse-closed, then $X_S:=X\cap \nu^{-1}(S)$ is an inverse-closed subset of $X$. In particular, if $x'\in X'$ is 
an involution, then $X_{\{x'\}}:=X\cap\nu^{-1}(x')$ is an inverse-closed subset of $X$.
\end{proposition}
\begin{proof}
Since $\nu$ is an epimorphism, $\nu^{-1}(S)=\{g\in G\mid \nu(g)\in S\}$. Because $S$ is inverse-closed, $\nu(g)\in S$ 
iff $\nu(g^{-1})=\nu(g)^{-1}\in S$, so $g\in\nu^{-1}(S)$ iff $g^{-1}\in\nu^{-1}(S)$. Since $X$ is inverse-closed, 
$X_S:=X\cap\nu^{-1}(S)$ is also inverse-closed.
\end{proof}

The following result is extracted from the classification  of regular Cayley maps 
of skew-type $1$ and $3$ on dihedral groups $\mathrm{D}_{2n}$~\cite{KMM2013, WF2005, Zhang2015a}.

\begin{proposition}[\cite{KMM2013, WF2005, Zhang2015a}]\label{D2n}
Let $M=\CM(\mathrm{D}_{2n}, X,p)$ be a regular Cayley map on the dihedral group
$\mathrm{D}_{2n}=\langle a, b\mid a^n=b^2=(ab)^2=1\rangle$.
Up to map isomorphism, we have
\begin{enumerate}[\rm(a)]
\item For $n=3$:
\begin{itemize}
\item[\rm(a.1)] if $M$ has skew-type $1$, then either $p=(b,ab)$ or $p=(b,ab,a^2b)$;
\item[\rm(a.2)] if $M$ has skew-type $3$, then $p=(a,a^{-1},ab,a^{-1}b)$.
\end{itemize}
\item For $n=6$:
\begin{itemize}
\item[\rm(b.1)] if $M$ has skew-type $1$, then either $p=(b,ab)$ or $p=(b,ab,a^2b,a^3b,a^4b,a^5b)$;
\item[\rm(b.2)] if $M$ has skew-type $3$, then   either $p=(a,a^{-1},ab,a^{-1}b)$ or $p=(a^3,ab,a^5,a^3b,a,a^5b)$.
\end{itemize}
\end{enumerate}
\end{proposition}

%%%%%%%%%%%%%%%%%%%
\section{Characterization lemmas}
From now on, we concentrate on regular Cayley maps of skew-type $3$ on the semidihedral group $\SD_{8n}$ . 
In this section, we first analyze the subgroup structure and automorphisms of $\SD_{8n}$, and then give a 
characterization of the core of skew morphisms corresponding to regular Cayley maps  of skew-type $3$ on $\SD_{8n}$.

Recall that $\SD_{8n}$ is defined by the presentation
\[
\SD_{8n}=\langle a,b\mid a^{4n}=b^2=1,\; bab^{-1}=a^{2n-1}\rangle,\quad n\geq 2.
\]
Since $[\SD_{8n}:\langle a\rangle]=2$, the elements of $\SD_{8n}$ are partitioned into
 elements of $a$-type and $b$-type, which belong to the cosets 
$\langle a\rangle$ and $\langle a\rangle b$, respectively. 
The order of a $b$-type element $a^ib$ is either $2$ or $4$; more precisely, 
\begin{equation}\label{Order}
|g|=2\ \text{if $g\in \langle a^2\rangle b$},\quad\text{and}\quad \qquad |g|=4\ \text{if $g\in \langle a^2\rangle ab$}.
\end{equation}
For all $n\geq 2$, it is well known that the automorphism group $\Aut(\SD_{8n})$ consists of mappings  
$\delta_{u,v}$ with $\delta_{u,v}(a)= a^u$ and $\delta_{u,v}( b)=a^v b$, where  $\gcd(u,4n)=1$ and  $v$ is even.
It is clear that the  cyclic subgroup $\langle \delta_{1,2}\rangle$ of $\Aut(\SD_{8n})$ is transitive 
on both $\langle a^2\rangle b$ and  $\langle a^2\rangle ab$. 

Moreover, if $m$ is a positive divisor of $4n$, then $N:=\langle a^m\rangle\ch\langle a\rangle\ch \SD_{8n}$, 
so $N\ch \SD_{8n}$. Thus,  for each $\delta\in\Aut(\SD_{8n})$, 
we may define an automorphism $\delta^*$ of $\SD_{8n}/N$ by $\delta^*(gN)=\delta(g)N$ $(g\in \SD_{8n})$. 
It follows that the mapping  $\Phi:\delta\mapsto\delta^*$
is a well defined homomorphism from $\Aut(\SD_{8n})$ into $\Aut(\SD_{8n}/N)$,
which we describe in greater detail as follows.
\begin{lemma}\label{auto}
Let $G:=\SD_{8n}$ be the semidihedral group defined by \eqref{Pre},  and $N=\langle a^m\rangle$.
\begin{enumerate}[\rm(a)]
 \item If $ m\geq 3$ is an odd divisor of $2n$, then $[\Aut(G/N):\Phi(\Aut(G))]=1$. 
 \item If $ m\geq 3$ is an even divisor of $2n$, then $[\Aut(G/N):\Phi(\Aut(G))]=2$.
 \end{enumerate}
\end{lemma}
\begin{proof}
Since $N=\langle a^m\rangle\ch \langle a\rangle\ch G$, we have $N\ch G$, so 
the above mapping $\Phi:\delta\mapsto\delta^*$ is
a homomorphism from $\Aut(G)$ into $\Aut(G/N)$. 
Note that the hypothesis implies that  $G/N$ is a dihedral group  defined by the presentation
\[
G/N=\langle \bar a,\bar b\mid \bar a^m=\bar b^2=1, \bar b\bar a\bar b^{-1}=\bar a^{-1}\rangle\cong\mathrm{D}_{2m}.
\]
Since $m\geq 3$,  every $\tau\in\Aut(G/N)$ has the form $\tau=\tau_{u',v'}:\bar a\mapsto \bar a^{u'}, \bar b\mapsto \bar a^{v'}\bar b$
 for some $u'\in\mathbb{Z}_m^*$ and $v'\in\mathbb{Z}_m$.

The number $u'\in\mathbb{Z}_m^*$ can be lifted to $u\in\mathbb{Z}_{4n}^*$ with $u\equiv u'\pmod{m}$~\cite[Lemma 1]{HNW2015}. 
If $v'$ is even, then we may take $v=v'$.  If  $m$ is odd and $v'$ is odd, then we can 
take $v=v'+m$, so that $\Phi(\delta_{u,v})=\delta_{u,v}^*=\tau_{u',v'}$, and hence $[\Aut(G/N):\Phi(\Aut(G))]=1$.
On the other hand, if $m$ is even and $v'$ is odd, then there exists no even number $v\in\mathbb{Z}_{4n}$ with
 $v\equiv v'\pmod{m}$. However, if $u',u''\in\mathbb{Z}_{m}^*$, and $v',v''\in\mathbb{Z}_m$ are odd,
then $u'u''\in\mathbb{Z}_m^*$ and $v'+v''\in\mathbb{Z}_m$ is even, so we may choose 
$u\in\mathbb{Z}_{4n}^*$ with $u\equiv u'u''\pmod{m}$ and $v:=v'+v''\in\mathbb{Z}_{4n}$ so that
$\delta_{u,v}\in\Aut(G)$ and  $\tau_{u',v'}\tau_{u'',v''}=\tau_{u'u'', v'+v''}=\delta_{u,v}^*$. 
Therefore, $[\Aut(G/N):\Phi(\Aut(G))]=2$, as required.
\end{proof}

Now we turn to subgroups of index $3$ in $\SD_{8n}$. It is evident that $\SD_{8n}$ has a subgroup of 
index $3$ only if $n$ is divisible by $3$.

\begin{lemma}\label{Index3}
For any positive integer $n$ divisible by $3$, the semidihedral group $\SD_{8n}$ contains exactly three 
distinct subgroups of index $3$, namely,
\[
H_1=\langle a^3, b \rangle,\quad H_2=\langle a^3, a^{-2}b \rangle,\quad H_3=\langle a^3, a^2b \rangle.
\]
Furthermore, $\delta_{1,2}(H_1)=H_3$ and $\delta_{1,-2}(H_1)=H_2$.
\end{lemma}
\begin{proof}
Write $n=3m$. Suppose $H$ is a subgroup of $\SD_{8n}$ with $[\SD_{8n}:H]=3$. Then $|\SD_{8n}|=24m$, 
$|\langle a\rangle|=12m$, and $|H|=8m$. Clearly $H\not\subseteq \langle a\rangle \lhd \SD_{8n}$ and 
$\SD_{8n}=H\langle a\rangle$. Now
\[
24m = |\SD_{8n}| = \frac{|H|\,|\langle a\rangle|}{|H\cap\langle a\rangle|} = \frac{8m\cdot12m}{|H\cap\langle a\rangle|},
\]
so $|H\cap\langle a\rangle|=4m$, and hence $H\cap\langle a\rangle=\langle a^3\rangle$. Since $[\SD_{8n}:H]=3$, 
$H$ must contain a $b$-type element $a^sb$. Because $[\SD_{8n}:\langle a^3,a^sb\rangle]=3$, we have 
\[
H=\langle a^3,a^sb\rangle=
\begin{cases}
\langle a^3,b\rangle, & \text{if } s\equiv0\pmod{3},\\
\langle a^3,a^{-2}b\rangle, & \text{if } s\equiv1\pmod{3},\\
\langle a^3,a^2b\rangle, & \text{if } s\equiv-1\pmod{3}.
\end{cases}
\]
These three subgroups are clearly distinct, as required.
\end{proof}

\begin{lemma}\label{normalsub}
Suppose $N$ is a normal subgroup of $\SD_{8n}$ of index at least $3$.
\begin{enumerate}[\rm(a)]
\item If $n$ is even, then $N< \langle a\rangle$.
\item If $n$ is odd, then either $N< \langle a\rangle$, or $N=\langle a^4,b\rangle$, or $N=\langle a^{4}, a^2b\rangle$.
\end{enumerate}
\end{lemma}
\begin{proof}
By hypothesis, $N\lhd G:=\SD_{8n}$ and $|G:N|\geq3$. Every proper subgroup of $\langle a\rangle$ is a 
normal subgroup of $G$ of index at least $3$. Now assume $N\not\leq\langle a\rangle$, so $N$ contains a 
$b$-type element $a^s b$. Then 
\[
a^{2-2n}=a(a^sb)a^{-1}(a^sb)^{-1}\in N\quad\text{and}\quad a^4=(a^{s+4}b)(a^{s}b)^{-1}=a^2(a^sb)a^{-2}(a^{s}b)^{-1}\in N.
\]

If $n$ is even, then $|a^{2-2n}|=2n$ and $N\geq\langle a^{2-2n},a^sb\rangle$, so $[G:N]\leq[G:\langle a^{2-2n},a^sb\rangle]=2$, a contradiction.

If $n=2n_1+1$ is odd, then $|a|=4(2n_1+1)$ and $a^{2n-2}=a^{4n_1}$. Since $\gcd(n_1,2n_1+1)=1$, 
we have $G'=\langle a^4\rangle\leq N$. Then $N/G'\leq G/G'$ and $[G/G':N/G']\geq3$. Since 
$G/G'\cong\mathbb{Z}_2\times\mathbb{Z}_4$, we have $N/G'=1$, $\langle\bar a^2\rangle$, $\langle\bar b\rangle$, 
or $\langle\bar a^2\bar b\rangle$, so $N=\langle a^4\rangle$, $\langle a^2\rangle$, $\langle a^4,b\rangle$, 
or $\langle a^4,a^2b\rangle$. Since $N\not\leq\langle a\rangle$, we get $N=\langle a^4,b\rangle$ or $\langle a^4,a^2b\rangle$,
as required.
\end{proof}

The following lemma describes the core of skew morphisms  of skew-type $3$ on $\SD_{8n}$.
\begin{lemma}\label{lemma Core}
Let $\varphi$ be a skew morphism on $\SD_{8n}$.
If $\varphi$ has skew-type $3$, then $n=3m$ is a multiple of $3$, and $\Core\varphi = \langle a^3 \rangle$ 
or $\langle a^6 \rangle$ (the latter only when $m=1$). More precisely, $\Core\varphi = \langle a^3 \rangle$ 
if $m\geq2$, and $\Core\varphi = \langle a^3 \rangle$ or $\langle a^6 \rangle$ if $m=1$.
\end{lemma}
\begin{proof}
Since $\varphi$ has skew-type $3$, $[\SD_{8n}:\Ker \varphi]=3$, so $n$ is a multiple of $3$. Write $n=3m$. 
By Lemma~\ref{Index3}, $\Ker \varphi\cap\langle a\rangle = \langle a^3 \rangle$. Since $\Core\varphi \leq \Ker \varphi$
 is a normal subgroup of $\SD_{8n}$, by Lemmas~\ref{Index3} and~\ref{normalsub} we have $\Core\varphi \leq \langle a^3 \rangle$.

On the other hand, $\varphi|_{\Ker\varphi}$ is an isomorphism from $\Ker\varphi$ to $\varphi(\Ker\varphi)$, 
so 
\[
\varphi(\Ker\varphi\cap\langle a\rangle) \cong \Ker\varphi\cap\langle a\rangle = \langle a^3\rangle.
\]

If $m\geq2$, then $\SD_{8n}$ has a unique cyclic group of order $4m$, namely $\langle a^3\rangle$. 
Hence $\varphi(\Ker\varphi\cap\langle a\rangle)=\Ker\varphi\cap\langle a\rangle=\langle a^3\rangle$, 
so $\langle a^3\rangle \leq \Core\varphi$. Thus $\Core\varphi = \langle a^3\rangle$ when $m\geq2$.

If $m=1$, then $|a^3|=4$. By \eqref{Order}, a $b$-type element $a^ib\in\SD_{24}$ has order $4$ iff $i$ is odd. 
Either $\varphi(\Ker\varphi\cap\langle a\rangle)=\Ker\varphi\cap\langle a\rangle=\langle a^3\rangle$,
 or $\varphi(\Ker\varphi\cap\langle a\rangle)=\langle a^i b\rangle$ for some odd $i$. In the first case 
 $\Core\varphi=\langle a^3\rangle$; in the second, since $\langle a^i b\rangle=\{1,a^6,a^ib,a^{6+i}b\}$, 
 we have $\varphi(a^6)=a^6$, so $\Core\varphi=\langle a^6\rangle$, as required.
\end{proof}

\begin{lemma}\label{prop:quotient-skewtype}
Let $\varphi$ be a Cayley skew morphism on $\SD_{8n}$ of skew-type $3$, and let $\overline{\varphi}$ be 
the quotient on $\SD_{8n}/\Core\varphi$. Then $\overline{\varphi}$ also has skew-type $3$.
\end{lemma}
\begin{proof}
By Proposition~\ref{quotient}, the skew-type of $\overline{\varphi}$ divides $3$, so it is either $1$ or $3$.
Suppose that $\overline{\varphi}$ has skew-type $1$, so $\overline{\varphi}\in\Aut(\overline{G})$ where $\overline{G}=\SD_{8n}/\Core\varphi$. 
By Lemma~\ref{lemma Core}, $\Core\varphi=\langle a^3\rangle$ or $\langle a^6\rangle$. Since $n=3m$ and $m\geq 1$,
 we have $a^{2n}\in\Core\varphi$. Note that $Z=\langle a^{2n}\rangle$ is the unique central subgroup of order $2$
  contained in $\langle a\rangle$.
Since the restriction of $\varphi$ to $\Core\varphi$ is an automorphism, $Z$ is a $\varphi$-invariant normal subgroup
of $G$, inducing a Cayley skew-morphism $\tilde\varphi$ on the dihedral group $\SD_{8n}/Z\cong\mathrm{D}_{4n}$. 
  Now the generating orbit $p$ of $\varphi$ projects onto a generating orbit $\tilde p$ of $\tilde\varphi$. Since $\overline{\varphi}$
  is an automorphism, $\bar p$ consists of only $b$-type elements, which implies that 
  $p$ also consists of only $b$-type elements. Therefore, $\tilde p$ consists of $b$-type elements too. Therefore, $\tilde\varphi$
  is an automorphism of $\SD_{8n}/Z$. Since $|Z|=2$, $|p|=|\tilde p|$ or $2|\tilde p|$. It follows that the skew-type
  of $\varphi$ is at most $2$, a contradiction.
\end{proof}

%%%%%%%%%%%%%%%%%%%%%%%%%%%%Construction%%%%%%%%%%%%%%%%%%%%%%%%%%%%%%
%%%%%%%%%%%%%%%%%%%%%%%%%%%%Construction%%%%%%%%%%%%%%%%%%%%%%%%%%%%%%
\section{Explicit Constructions}
In this section, we show that the Cayley maps given in Theorem~\ref{Main}
are indeed regular Cayley maps of skew-type $3$ on the semidihedral groups $\SD_{8n}$.

Recall that two families of Cayley maps on  $\SD_{8n}$ have been presented in Theorem~\ref{Main}: The first family
$\CM(\SD_{24},X,p)$  is  exceptional, which exists only when $n=3$, while the second family 
$M(n,t)=\CM(\SD_{8n},X,p)$ is infinite, exists for all positive integers $n$ divisible by $3$, and 
 integers $t\in\mathbb{Z}_{4n}^*$ of odd multiplicative order with $t\equiv1\pmod{6}$.

 Note that when $n=3$, there is a unique Cayley map $M(3,1)$
 on $\SD_{24}$ in the second family. We shall denote this map by $\CM(\SD_{24}, X', p')$,
 where $X'=\{a,a^{-1},a^4b,a^8b\}$ and $p'=(a,a^{-1},a^4b,a^8b)$.

 The following theorem show that these two Cayley maps on $\SD_{24}$ are indeed non-isomorphic
 regular Cayley maps of skew-type $3$.

%%%%%%%%%%%%%%
\begin{theorem} \label{Class1}
The maps $\CM(\SD_{24},X,p)$ and $\CM(\SD_{24},X',p')$ are nonisomorphic regular 
Cayley maps of skew-type $3$ on the
semidihedral groups $\SD_{24}$.
\end{theorem}
\begin{proof}
The main idea to prove this theorem is the following: For $p$ and $p'$, we first construct 
a permutation $\varphi$ on $\SD_{24}$ which
 contains an orbit consistent with  the $p$ and $p'$, and then verify that it is a 
skew morphism of skew-type $3$. 

For  $p$, we define  $ \varphi $ on $ \SD_{24} $ as:
	\[
	\varphi(a^i b^j) = 
	\begin{cases} 
		a^{i-2+6j}b^{j+1}, & \text{if } i \text{ is odd}, \\
		a^{-i+6j}b^j, & \text{if } i \text{ is even}.
	\end{cases}
	\]
It is easily seen that 	$ \varphi $ preserves $ X $ 
	and the restriction of $ \varphi $ to $ X $ is $p$. Now, by the definition of $\varphi$, if $i$ is even, then
\begin{align}\label{cubic1}
\varphi(a^3 a^i b^j)=a^{i+1+6j}b^{j+1}=ab\cdot a^{-i+6j}b^j =\varphi(a^3)\varphi(a^i b^j);
\end{align}
if $i$ is odd, then 
\begin{align}\label{cubic2}
\varphi(a^3 a^i b^j)=a^{-3-i+6j}b^j=ab \cdot a^{i-2+6j}b^{j+1}  = \varphi(a^3)\varphi(a^i b^j).
\end{align}
Similarly, we have $\varphi(a^6 a^i b^j)=\varphi(a^6)\varphi(a^i b^j)$  and
	$\varphi(a^9 a^i b^j)=\varphi(a^9)\varphi(a^i b^j)$.
	
Moreover, if $i$ is even, then 
\begin{equation}\label{even}
\begin{aligned}
\varphi(a a^i b^j)&=a^{i-1+6j}b^{j+1}= a^{11}b\cdot a^{-i+6j}b^j = \varphi(a)\varphi^9(a^i b^j),\\
\varphi(b a^i b^j) &=a^{-5i+6+6j}b^{j+1} =a^6b\cdot a^{-i+6j}b^j = \varphi(b)\varphi^5(a^i b^j),\\
\varphi(a^{-1} a^i b^j)&=a^{i-3+6j}b^{j+1} =a^9b\cdot a^{-i+6j}b^j = \varphi(a^{-1})\varphi^5(a^i b^j);
\end{aligned}
\end{equation}
if $i$ is odd, then
\begin{equation}\label{odd}
\begin{aligned}
\varphi(a a^i b^j)&=a^{-1-i+6j}b^j =a^{11}b \cdot a^{i+6+6j}b^{j+1} = \varphi(a)\varphi^9(a^i b^j),\\
\varphi(b a^i b^j)&=a^{5i+4+6j}b^j = a^6b \cdot a^{i+2+6j}b^{j+1}=\varphi(b)\varphi^5(a^i b^j),\\
\varphi(a^{-1} a^i b^j)&=a^{1-i+6j}b^j =a^9b \cdot a^{i+2+6j}b^{j+1}=\varphi(a^{-1})\varphi^5(a^i b^j).  
\end{aligned}
\end{equation}
Taking the three possibilities $i\equiv0,1,-1\pmod{3}$ into
account and using the identities from \eqref{cubic1}--\eqref{odd}, it is now routine to verify that
$ \varphi $ is a skew morphism of $ \SD_{24} $ with the associated power  power function $ \pi $ given by
\[
\pi(g) = 
	\begin{cases} 
1, & \text{if $g\in \langle a^3,a^2b\rangle$,}\\ 
5,& \text{if $g\in\langle a^3,a^2b\rangle a^{-1}$,}\\
9, & \text{if $g\in\langle a^3,a^2b\rangle a$}.
	\end{cases}
	\]
Thus, $\Ker \varphi=\langle a^3,a^2b\rangle$
and the Cayley map $ M= \CM(\SD_{24}, X, p) $ is regular of skew-type~3.

For $p' = (a, a^{-1}, a^4b, a^8b)$, it is trivial to verify that the
permutation $\varphi$ defined by
\[
	\varphi(a^i b^j) = 
	\begin{cases} 
		a^{i+6j}b^j, & \text{if } i \equiv 0 \pmod{3}, \\
		a^{i-2+6j}b^j, & \text{if } i \equiv 1 \pmod{3}, \\
		a^{i+5}b^{j+1}, & \text{if } i \equiv -1 \pmod{3}
	\end{cases}
	\]
is a desired skew morphism. 
\end{proof}

In what follows we deal with the remaining case  $n=3m\geq 6$.
%%%%%%%%%%%%%%%

\begin{theorem}\label{Class2}
The maps $M(n;t)=\CM(\SD_{8n},X,p)$ $(n\geq 6)$ given in Theorem~\ref{Main}(b) are regular
Cayley maps of skew-type $3$ on the semidihedral groups $\SD_{8n}$.
\end{theorem}
%%%%%%%%%%%%%
\begin{proof} 
We define a permutation on $\SD_{8n}$ as
\begin{align*}
	\varphi(a^i b^j) =
	\begin{cases}
		a^{(2m-t) i + 3tj + 3j} b^j ,& i \equiv 0 \pmod{3}, \\
		a^{(2m-t) i + 3tj + 3j - 2m} b^j ,& i \equiv 1 \pmod{3}, \\
		a^{(2m-t) i + 3tj - 3j + 2m + 3} b^{j+1} ,& i \equiv -1 \pmod{3}.
	\end{cases}
\end{align*}
Recall that $t\equiv1\pmod{6}$ and the multiplicative order $k=o_{4n}(t)$ of  $t\in\mathbb{Z}_{4n}$
 is odd. Using induction on $l$, we derive the following
 identities:
\begin{align*}
	\varphi^{4l}(a^i b^j) &= a^{it^{4l} - 3t^{4l} j + 3j} b^j,\\
	\varphi^{4l+1}(a^i b^j) &= 
	\begin{cases}
		a^{(2m - t^{4l+1})i + 3t^{4l+1} j + 3j} b^j ,& i \equiv 0 \pmod{3}, \\
		a^{(2m - t^{4l+1})i + 3t^{4l+1} j + 3j - 2m} b^j ,& i \equiv 1 \pmod{3}, \\
		a^{(2m - t^{4l+1})i + 3t^{4l+1} j -3j + 2m + 3} b^{j+1} ,& i \equiv -1 \pmod{3},
	\end{cases}\\
	\varphi^{4l+2}(a^i b^j) &= 
	\begin{cases} 
		a^{it^{4l+2} - 3t^{4l+2} j + 3j} b^j ,& i \equiv 0 \pmod{3}, \\
		a^{it^{4l+2} - 3t^{4l+2} j - 3j + 3} b^{j+1} ,& i \equiv 1 \pmod{3}, \\
		a^{it^{4l+2} - 3t^{4l+2} j - 3j + 6m + 3} b^{j+1} ,& i \equiv -1 \pmod{3},
	\end{cases}\\
	\varphi^{4l+3}(a^i b^j) &= 
	\begin{cases}
		a^{(2m - t^{4l+3})i + 3t^{4l+3} j + 3j} b^j ,& i \equiv 0 \pmod{3}, \\
		a^{(2m - t^{4l+3})i + 3t^{4l+3} j - 3j + 4m + 3} b^{j+1} ,& i \equiv 1 \pmod{3}, \\
		a^{(2m - t^{4l+3})i + 3t^{4l+3} j + 3j + 2m} b^j ,& i \equiv -1 \pmod{3},
	\end{cases}
\end{align*}
Using these identities, it is routine to verify that $\varphi$  contains $X$ as an orbit, and 
 the restriction of $\varphi$ to $X$ is  the given permutation $p$ on $X$.

Moreover, by the definition of $\varphi$, for any integer $l$, we have 
\begin{align}
\varphi(a^{3l} a^i b^j) &= 
\begin{cases} 
    a^{(2m-t) (i+3l) + 3tj + 3j} b^j, & i \equiv 0 \pmod{3},\\
    a^{(2m-t) (i+3l) + 3tj + 3j - 2m} b^j, & i \equiv 1 \pmod{3},\\
    a^{(2m-t) (i+3l) + 3tj - 3j + 2m + 3} b^{j+1}, & i \equiv -1 \pmod{3},
\end{cases} \notag \\
&= 
\begin{cases}
    a^{(2m-t) 3l} a^{(2m-t) i + 3tj + 3j} b^j, & i \equiv 0 \pmod{3},\\
    a^{(2m-t) 3l} a^{(2m-t) i + 3tj + 3j - 2m} b^j, & i \equiv 1 \pmod{3},\\
    a^{(2m-t) 3l} a^{(2m-t) i + 3tj - 3j + 2m + 3} b^{j+1}, & i \equiv -1 \pmod{3},
\end{cases} \notag \\
&= \varphi(a^{3l}) \varphi(a^i b^j), \label{cubic2}
\end{align}
and
\begin{align*}
\varphi^{2k+1}(a^ib^j)&=\varphi^{4(k-1)/2+3}(a^ib^j)\\&=
\begin{cases}
	a^{(2m-t^{2k+1})i+3t^{2k+1}j+3j}b^j, &i \equiv 0 \pmod{3}\\
	a^{(2m-t^{2k+1})i+3t^{2k+1}j-3j+4m+3}b^j, &i \equiv 1 \pmod{3}\\
	a^{(2m-t^{2k+1})i+3t^{2k+1}j+3j+2m}b^j, &i \equiv -1 \pmod{3}
\end{cases}\\&=
\begin{cases}
	a^{(2m-t)i+3tj+3j}b^j, &i \equiv 0 \pmod{3}\\
	a^{(2m-t)i+3tj-3j+4m+3}b^j, &i \equiv 1 \pmod{3}\\
	a^{(2m-t)i+3tj+3j+2m}b^j, &i \equiv -1 \pmod{3}.
\end{cases}
\end{align*}
Therefore, 

\begin{align}
	\varphi(aa^ib^j)&=
	\begin{cases}
		a^{[(i+1)t+(3-3t)j-1](2m-1)-1+6j}b^j, &i \equiv 0 \pmod{3}\\
		a^{[(i+1)t+(3-3t)j+1](2m-1)+4}b^{j+1}, &i \equiv 1 \pmod{3}\\
		a^{[(i+1)t+(3-3t)j](2m-1)+6j}b^j, &i \equiv -1 \pmod{3}
	\end{cases}\notag
	\\&=
	\begin{cases}
		a^{-t} a^{(2m-t)i+3tj+3j}b^j, &i \equiv 0 \pmod{3}\\
		a^{-t} a^{(2m-t)i+3tj-3j+4m+3}b^j, &i \equiv 1 \pmod{3}\\
		a^{-t} a^{(2m-t)i+3tj+3j+2m}b^j, &i \equiv -1 \pmod{3}
	\end{cases}\notag
	\\&=\varphi(a)\varphi^{2k+1}(a^ib^j)\label{atype}
\end{align}
and 
\begin{align}
\varphi(ba^ib^j)&=\varphi(a^{(6m-1)i}b^{j+1}) \notag\\
&= \begin{cases}
    a^{[(6m-1)it+3-3t-3j+3tj](2m-1)+6-6j}b^{j+1}, & i \equiv 0 \pmod{3}, \\
    a^{[(6m-1)it+3-3t-3j+3tj+1](2m-1)+4}b^j, & i \equiv 1 \pmod{3}, \\
    a^{[(6m-1)it+3-3t-3j+3tj-1](2m-1)-1+6-6j}b^{j+1}, & i \equiv -1 \pmod{3},
\end{cases}  \notag\\
&= \begin{cases}
    a^{3t+3}b  a^{(2m-t)i+3tj+3j}b^j, & i \equiv 0 \pmod{3}, \\
    a^{3t+3}b  a^{(2m-t)i+3tj-3j+4m+3}b^j, & i \equiv 1 \pmod{3}, \\
    a^{3t+3}b  a^{(2m-t)i+3tj+3j+2m}b^j, & i \equiv -1 \pmod{3},
\end{cases} \notag \\
&= \varphi(b)\varphi^{2k+1}(a^ib^j). \label{btype}
\end{align}

Similarly, if $k \equiv 1 \pmod{4}$, then
\begin{align*}
	\varphi^{k+1}(a^ib^j)&= \varphi^{4(k-1)/4+2}(a^ib^j)\\
	&=\begin{cases}
	 a^{it^{k+1}-3t^{k+1}j+3j}b^j,  &i \equiv 0 \pmod{3},\\
	 a^{it^{k+1}-3t^{k+1}j-3j+3}b^{j+1}, &i \equiv 1 \pmod{3},\\
	 a^{it^{k+1}-3t^{k+1}j-3j+6m+3}b^{j+1},  &i \equiv -1 \pmod{3},
	\end{cases}\\&=
	\begin{cases}
		a^{it-3tj+3j}b^j,  &i \equiv 0 \pmod{3},\\
		a^{it-3tj-3j+3}b^{j+1}, &i \equiv 1 \pmod{3},\\
		a^{it-3tj-3j+6m+3}b^{j+1},  &i \equiv -1 \pmod{3}.
	\end{cases}
\end{align*}
If $k \equiv 3 \pmod{4}$, then
\begin{align*}
	\varphi^{3k+1}(a^ib^j)&= \varphi^{4(3k-1)/4+2}(a^ib^j)\\
	&=
	\begin{cases}
		a^{it^{k+1}-3t^{k+1}j+3j}b^j,  &i \equiv 0 \pmod{3},\\
		a^{it^{k+1}-3t^{k+1}j-3j+3}b^{j+1}, &i \equiv 1 \pmod{3},\\
		a^{it^{k+1}-3t^{k+1}j-3j+6m+3}b^{j+1},  &i \equiv -1 \pmod{3},
	\end{cases}\\&=
	\begin{cases}
		a^{it-3tj+3j}b^j,  &i \equiv 0 \pmod{3},\\
		a^{it-3tj-3j+3}b^{j+1}, &i \equiv 1 \pmod{3},\\
		a^{it-3tj-3j+6m+3}b^{j+1},  &i \equiv -1 \pmod{3}.
	\end{cases}
\end{align*}
Since 
\begin{align*}
	\varphi(a^{-1}a^ib^j)&=
	\begin{cases}
		a^{[(i-1)t+(3-3t)j+1](2m-1)+4}b^{j+1}, &i \equiv 0 \pmod{3},\\
		a^{[(i-1)t+(3-3t)j](2m-1)+6j}b^j, &i \equiv 1 \pmod{3},\\
		a^{[(i-1)t+(3-3t)j-1](2m-1)-1+6j}b^j, &i \equiv -1 \pmod{3},
	\end{cases}\\&=
	\begin{cases}
		a^{t+3}b  a^{it-3tj+3j}b^j, &i \equiv 0 \pmod{3},\\
		a^{t+3}b  a^{it-3tj-3j+3}b^{j+1}, &i \equiv 1 \pmod{3},\\
		a^{t+3}b  a^{it-3tj-3j+6m+3}b^{j+1}, &i \equiv -1 \pmod{3},
	\end{cases}
\end{align*}
we get
\begin{equation}\label{a-1type}
\begin{aligned}
	\varphi(a^{-1}a^ib^j)=
	\begin{cases}
		\varphi(a^{-1})\varphi^{k+1}(a^ib^j), &k\equiv 1 \pmod{4},\\
		\varphi(a^{-1})\varphi^{3k+1}(a^ib^j), &k\equiv 3 \pmod{4}.
	\end{cases}
\end{aligned}
\end{equation}
It follows from the identities \eqref{cubic2}--\eqref{a-1type}  that
 $\varphi$ is a skew-morphism of $\SD_{8n}$ with the associated power function $\pi$
 determined by the following formula: If $k\equiv1\pmod{4}$, then
  \[
 \pi(g)=
 \begin{cases}
 1,& \text{if $g\in \langle a^3, ab \rangle$,}\\
 2k+1, &\text{if $g\in \langle a^3, ab \rangle a$,}\\
 k+1, &\text{if $g\in \langle a^3, ab \rangle a^{-1}$;}
  \end{cases}
 \]
 if $k\equiv3\pmod{4}$, then
  \[
 \pi(g)=
 \begin{cases}
 1,& \text{if $g\in \langle a^3, ab \rangle$,}\\
 2k+1, &\text{if $g\in \langle a^3, ab \rangle a$,}\\
  3k+1, &\text{if $g\in \langle a^3, ab \rangle a^{-1}$}.
 \end{cases}
 \]
Therefore, $\Ker \varphi=\langle a^3, ab \rangle$, and $M=\CM(\SD_{8n},X,p)$ is regular and skew-type $3$.

Finally, suppose that $M(n,t)=\CM(\SD_{8n},X,p)$ and $M(n,t')=\CM(\SD_{8n},X',p')$
are two maps corresponding to the parameters $t$ and $t'$ as described in Theorem~\ref{Main}(b).
If $M(n,t) \cong M(n,t')$, then there exists an automorphism $\delta $ of $\SD_{8n}$ such that $\delta(X)=X' $ 
and $\delta p=p'\delta$. Since $\langle a \rangle$ is a characteristic subgroup of $\SD_{8n}$, we 
have $\delta(a)=a^u$ for some $u\in\mathbb{Z}_{4n}^*$. Then $\delta p(a)=p' \delta(a) $ implies 
$a^{-tu}=a^{-t'u}$, and hence $t=t'$. Therefore,
for fixed $n$, the map $M(n,t)$ is uniquely determined by the  parameter $t\in\mathbb{Z}_{4n}^*$ up to isomorphism. 
 \end{proof}

%%%%%%%%%%%%%%%%%%%%%%%%%%%%%%%%%%%%%%%%%%%%%%%%%%%%
%%%%%%%%%%%%%%%%%%%%%%%%%%%%%% Classification %%%%%%%%%%%
\section{Classification Results}
In this section, we show that every regular Cayley map of skew-type $3$ on the semidihedral group $\SD_{8n}$
is isomorphic to some map given in Theorem~\ref{Main}. Thereby, combining with
Theorem~\ref{Class1} and \ref{Class2} we will complete the classification.

The main idea for the proof is the covering techniques introduced in Section 2. We explain it briefly as follows: 
Suppose that $M=\CM(\SD_{8n},X,p)$ is a regular Cayley map of skew-type $3$ on $\SD_{8n}$, and let
$\varphi$ be the corresponding Cayley skew morphism. By Lemma~ \ref{lemma Core}, $n=3m$ for some positive
integer $m$, either (a) $ \Core\varphi =\langle a^3 \rangle, \langle a^6 \rangle$ if $m=1$, 
or  (b) $\Core\varphi = \langle a^3 \rangle $ if $m\geq 2$. By Lemma~\ref{prop:quotient-skewtype},
 the quotient map $\overline{M}=\CM(\overline{\SD}_{8n},\overline{X},\bar p)$
induced by $\Core\varphi$ is a regular Cayley map of skew-type $3$ on a smaller group $\overline{\SD}_{8n}$.
 These quotient maps  are known (see Proposition~\ref{D2n} and Lemma~\ref{auto}): If $\Core\varphi = \langle a^3 \rangle $,
then  $\overline{\SD}_{8n}\cong \mathrm{D}_6$. By Lemma~\ref{auto}(a), every automorphism of $\overline{\SD}_{8n}$
can be induced by an automorphism of $\SD_{8n}$ via the canonical homomorphism 
$\Phi:\Aut(\SD_{8n})\to\Aut(\overline{\SD}_{8n}),\delta\mapsto\delta^*$. Thus, by Proposition~\ref{D2n}(a), we may assume
 \[
 \bar p=(\bar a,\bar a^{-1},\bar a\bar b,\bar a^{-1}\bar b).
 \]
If $\Core\varphi = \langle a^6 \rangle $, then  $\overline{\SD}_{8n}\cong\mathrm{D}_{12}$.
By Lemma~\ref{auto}(b), $[\Aut(\overline{\SD}_{8n}):\Phi(\Aut(\SD_{8n}))]=2$, 
so by Proposition~\ref{D2n}(b), we may assume
 \[
 \bar p=(\bar a^3,\bar a\bar b,\bar a^5,\bar a^3\bar b,\bar a,\bar a^5\bar b)\quad\text{or}\quad 
\bar p=(\bar a^3,\bar a^2\bar b,\bar a^5,\bar a^4\bar b,\bar a, \bar b).
 \]
 Note that the latter is obtained from the former by conjugation by 
 \[
 \bar\delta\in\Aut(\overline{\SD}_{8n})\backslash\Phi(\Aut(\SD_{8n})):\bar \delta:\bar a\mapsto \bar a,\bar b\mapsto \bar a\bar b.
 \]
For Case (a), by Proposition~\ref{Period}, $\varphi$ has period $4$, so $|X|=4k$ for some positive integer $k$, and
we may write
\begin{equation}\label{Perm}
p=(x_0,x_1,x_2,x_3,\ldots, x_{4i},x_{4i+1},x_{4i+2},x_{4i+3},\ldots, x_{4k-4},x_{4k-3},x_{4k-2},x_{4k-1}),
\end{equation}
where $\bar x_0=\bar a$. For each $i\in\{0,1,2,3\}$, we set  
\begin{equation}\label{Inverse}
X_i:=\{x_j\mid j\equiv i\pmod{4}\quad\text{and}\quad j\in\mathbb{Z}_{4k}\},
\end{equation}
so $X_i=X\cap\nu^{-1}(\overline{X}_i)$, where $\nu:\SD_{8n}\to\overline{\SD}_{4n}$ is the natrual epimorphism. 
Now we use the hypothesis
that $\varphi$ is a Cayley skew morphism of skew-type $3$ to determine the distribution-of-inverses function $\chi$,
the power function $\pi$, the permutation $p$, and finally, the skew morphism $\varphi$ itself. The 
method to deal with Case (b) is similar.

Due to technical reasons, we first deal with Case (a), where the group $\SD_{8n}=\SD_{24}$ is 
of the smallest order. 
%%%%%%%%%%%%%
\begin{theorem}\label{Main1}
Every regular Cayley map of skew-type $3$ on
the semidihedral group $\SD_{24}$ is isomorphic to one of the maps 
 constructed in Theorem~\ref{Class1}.
\end{theorem}
\begin{proof}
Suppose that $ M = \CM(\SD_{24},X,p)  $ is a regular Cayley map of skew-type 3.
As just explained, by Lemma \ref{lemma Core},  we have $\Core\varphi = \langle a^3 \rangle $ 
 or $ \langle a^6 \rangle$.
%%%%%%%%%%%%%%% Case A %%%%%%%%%%%%%%%%%%%%%%%%%
\begin{case}[A]$\Core\varphi = \langle a^3 \rangle$.\par

In this case, $\overline{M} = \CM(\overline{\SD}_{24}, \overline{X}, \overline{p})$
is a regular Cayley map of skew-type $3$ on the dihedral group $\overline{\SD}_{24}\cong \mathrm{D}_6$,
and we can set
\begin{align}
\label{Form2}
	\bar{\varphi} = (\bar{1})(\bar{b})(\bar{a}, \bar{a}^{-1}, \bar{a}\bar{b}, \bar{a}^{-1}\bar{b})
	\quad\text{and}\quad
	\bar{p} = (\bar{a}, \bar{a}^{-1}, \bar{a}\bar{b}, \bar{a}^{-1}\bar{b}).
\end{align}
Note that
\[
\bar{\pi}(\bar a)=3,\quad \bar{\pi}(\bar a^{-1})=\bar{\pi}(\bar a^{-1}\bar b)=2\quad\text{and}\quad \bar{\pi}( \bar a\bar b)=1.
\]
 Since $[\SD_{8n}:\Ker \varphi]=3$, 
by Lemma~\ref{Index3}, $\Ker \varphi = \langle a^3, ab \rangle$, $\langle a^3, a^{-1}b\rangle$, or $\langle a^3, b \rangle$.
But $\Ker \overline{\varphi}=\langle \bar a\bar b\rangle$, so $b,a^{-1}b\notin\Ker \varphi$, we get 
$\Ker \varphi = \langle a^3, ab \rangle$.

Now we assume that $|X| = 4k$ for some positive integer $k$, and use the notation defined in \eqref{Perm} and \eqref{Inverse}
We distinguish two subcases.
%%%%%%%%%%
\begin{subcase}[A.1]$k=1$.\par
 In this subcase, since both $\overline{X}_2=\bar a\bar b$ and $\overline{X}_3=\bar a^{-1}\bar b$ are
  involutions, $x_2$ and $x_3$ are also involutions (see Lemma~\ref{Lift}) which must be elements of 
  $b$-type. Since $X$ is an inverse-closed generating set of $\SD_{24}$, we have $x_0=x_1^{-1}$,
 and  $\SD_{24}=\langle x_0, x_2\rangle$, so $\langle x_0\rangle=\langle a\rangle$. 
 Up to map isomorphism,  we can assume $p=(a,a^{-1}, a^{4}b,a^{3j-1}b)$ for some odd number $j=1,3\in\mathbb{Z}_4$,
  so $p = (a, a^{-1}, a^4b, a^{2}b)$, or $ p = (a, a^{-1}, a^{4} b, a^8 b).$ Note that the first case
   cannot occur since 
   \[
  \varphi(a^4b\cdot a^2b)=\varphi(a^{14}b)=\varphi(a^2b)=a\neq a^7b=  a^2ba=\varphi(a^4b)\varphi(a^2b)=  \varphi(a^4b\cdot a^2b).
   \]
 Therefore, up to map isomorphism, we can take $p = (a, a^{-1}, a^4 b, a^8 b) $.
 \end{subcase}

 %%%%%%
 %%%%%%
\begin{subcase}[A.2]$ k\geq 2 $.\par
Take $ c \in X $ such that $ \bar{c} = \bar{a}. $ 
Let $i\in\mathbb{Z}_k$ be an arbitrary integer. Then
\begin{align}
&\bar{x}_{4i} = \bar{a},  &&\bar {x}_{4i+1}  = \bar{a}^{-1}, &&\bar{x}_{4i+2}  = \bar{a}\bar{b}, &&\bar{x}_{4i+3}  = \bar{a}^{-1}\bar{b};\label{Proj1}\\
&\bar\chi(\bar{x}_{4i}) \equiv 1, &&\bar\chi(\bar{x}_{4i+1}) \equiv  3,  &&\bar\chi(\bar{x}_{4i+2}) \equiv  0, &&\bar\chi(\bar{x}_{4i+3}) \equiv  0\pmod{4};\label{Proj2}\\
&\bar\pi(\bar{x}_{4i}) \equiv 3, &&\bar\pi(\bar{x}_{4i+1}) \equiv  2,  &&\bar\pi(\bar{x}_{4i+2}) \equiv  1, &&\bar\pi(\bar{x}_{4i+3}) \equiv  2\pmod{4}.\label{Proj1}
\end{align}
It follows that both $\pi(x_{4i+1})$ and $\pi(x_{4i+3})$ are even, and $\pi(x_{4i})\neq 1$ is odd.
Since $ab\in\Ker \varphi$, we get $x_{4i+2}\in\Ker \varphi$, so $\pi(x_{4i+2})=1$.
By hypothesis, $\pi$ takes three distinct values in $\mathbb{Z}_{4k}$, therefore, we must have
\begin{align}\label{Form3}
\pi(x_{4i+1})\equiv\pi(x_{4i+3})\quad\text{and}\quad \chi(x_{4i+2})\equiv\chi(x_{4i+3})\pmod{4k}.
\end{align}
By \eqref{Proj2}, we may assume 
\[
\chi(x_0) \equiv 4l_0+1 ,\quad \chi(x_{1}) \equiv 4l_1-1 ,\quad\text{and}\quad
\chi(\varphi^{2}(c))=\chi(\varphi^{3}(c))\equiv 4l \pmod{4k},
\]
for some integers $l_0,l_1,l\in\mathbb{Z}_k$. Thus, using \eqref{Dist}, we have
\begin{align*}
&\pi(b)=\pi(a)=\pi(c) \equiv (4l_1-1)-(4l_0+1)+1 \equiv 4(l_1-l_0)-1 \pmod{4k}, \\
&\pi(a^{-1})=\pi(x_{1}) \equiv 4l-(4l_1-1)+1 \equiv 4(l-l_1)+2 \pmod{4k}.
\end{align*}
Since $\pi(\varphi^i(b))= 4(l_1-l_0)-1$ for all $i$, by Lemma~\ref{Formula}(b), we then obtain
\begin{align*}
4(l-l_1)+2&\equiv\pi(a^{-1})\equiv \pi(a^{-1}b) 
\equiv \sum_{i=0}^{\pi(a^{-1})-1} \pi(\varphi^i(b))\\
&\equiv \sum_{i=0}^{4(l-l_1)+1} \pi(\varphi^i(b))\equiv (4(l_1-l_0)-1)(4(l-l_1)+2)\pmod{4k},
\end{align*}
which reduces to
\[
4(l_1-l_0)(l-l_1)-2(l-l_1)+2(l_1-l_0)-1\equiv0\pmod{k}.
\]
Thus, $k$ is odd. Since $2\leq k\leq|\Core\varphi|=4$, we get $k=3.$ Therefore, we may write
\[
p=(x_0,x_1,x_2,x_3,x_4,x_5,x_6,x_7,x_8,x_9,x_{10},x_{11}),
\]
 where $x_0=c$ and
\[
\bar{x_i}=\begin{cases}
\bar{a},& i\equiv0\pmod{4},\\
\bar{a}^{-1},& i\equiv1\pmod{4},\\
\bar{a}\bar{b},& i\equiv2\pmod{4},\\
\bar{a}^{-1}\bar{b},& i\equiv3\pmod{4}.
\end{cases}
\]
 Note that, for any $i$, both $x_{4i}$ and $x_{4i+1}$ are elements of $a$-type, while
  both $x_{4i+2}$ and $x_{4i+3}$ are elements of $b$-type; since $X$ is closed under inverses, by \eqref{Order}, we have
  either  $  |x_{4i+2}|=|x_{4i+3}|= 2$ or $  |x_{4i+2}|=|x_{4i+3}|= 4$.
 
  By Proposition~\ref{Lift}, the subsets  $X_2=\{x_2,x_6,x_{10}\}$ and $X_3=\{x_3,x_7,x_{11}\}$ are 
both inverse-closed, so it cannot happen that $  |x_{4i+2}|=|x_{4i+3}|= 4$ for all $i$.  
If there exists some $i$ such that $  |x_{4i+2}|=|x_{4i+3}|=  |x_{4i+6}|=|x_{4i+7}|=2,$
  then, by Lemma~\ref{CD}, the common difference $d$ of $\chi$ is $0$, and so $  |x_{4i+2}|=|x_{4i+3}|= 2$ for all $i$.
  It follows that $x_2,x_6,x_{10}\in\Ker \varphi$. However, by \eqref{Order}, we find that 
  $\Ker \varphi=\langle a^3,ab\rangle$ contains only two elements 
  of $b$-type, that is, $a^4b$ and $a^{10}b$, a contradiction.

  Therefore, without loss of generality, we may assume
  \[
  |x_2|=|x_3|=2\quad\text{and}\quad |x_6|=|x_7|=|x_{10}|=|x_{11}|=4.
  \]
We have $x_2^{-1}=x_2$, $x_3^{-1}=x_3$, $x_6^{-1}=x_{10}$ and $x_7^{-1}=x_{11}$. Thus, $d=\chi(x_6)-\chi(x_2)=4$.
By \eqref{Proj2}, we may set $\chi(x_0)=4l_0+1$ and $\chi(x_1)=4l_1-1$. Then
\[
\chi(x_4)=4l_0+5,\quad\chi(x_5)=4l_1+3,\quad\chi(x_8)=4l_0+9,\quad \chi(x_9)=4l_1+7.
\]
Using \eqref{Dist}, one may now compute $\pi(x_i)$ for all $i$, as shown below:
\[
\pi(x_i)=
\begin{cases}
4(l_1-l_0)-1,& i\equiv0\pmod{4},\\
-4l_1+2, &i\equiv 1\pmod{4},\\
1, &i\equiv 2\pmod{4},\\
\pi(x_3)=4l_0+6 &i\equiv 3\pmod{4}.
\end{cases}
\]
By \eqref{Form3}, we have $-4l_1+2=\pi(x_1)=\pi(x_3)=4l_0+6$, which reduces to 
\begin{equation}\label{EL1}
l_0+l_1\equiv2\pmod{3}.
\end{equation}
Note that $\pi(a^{-1})=\pi(x_1)=-4l_1+2$ and $\pi(b)=\pi(a)=\pi(x_0)=4(l_1-l_0)-1$, so by Lemma~\ref{Form3}(b), we have
\[
4l_0+6=\pi(a^{-1}b)=\sum_{i=0}^{\pi(a^{-1})-1}\pi(\varphi^i(b))=\pi(a^{-1})\pi(b)=(-4l_1+2)(4(l_1-l_0)-1)\pmod{12},
\]
which reduces $-16l_1(l_1-l_0)+12(l_1-l_0)\equiv8\pmod{12}$. Thus, we get
\begin{equation}\label{EL2}
l_1(l_1-l_0)\equiv 1\pmod{3}.
\end{equation}
Solving the system consisting of \eqref{EL1} and \eqref{EL2} gives a unique solution $(l_0, l_1)=(0,-1)\pmod{3}.$ It follows
that $\chi(x_0)=1$ and $\chi(x_1)=7$, so $x_0=x_1^{-1}=x_8$, a contradiction.
\end{subcase}
\end{case}

%%%%%%
\begin{case}[B]$m=1$ and $\Core\varphi = \langle a^6 \rangle$.\par
In this case, the quotient map $\bar M=\CM(\overline{\SD}_{24},\bar X,\bar p)$ is a regular Cayley 
map  of skew-type $3$ on the dihedral group $\overline{\SD}_{24} \cong \mathrm{D}_{12}.$
Thus, by Proposition~\ref{D2n} and Lemma~\ref{auto}(b), either 
\[
 \bar p=(\bar a^3,\bar a\bar b,\bar a^5,\bar a^3\bar b,\bar a,\bar a^5\bar b)
 \quad\text{or}\quad\bar p=(\bar a,\bar a^{-1},\bar a\bar b,\bar a^{-1}\bar b)
 \]
 or their conjugates:
 \[
  \bar p=(\bar a^3,\bar a^2\bar b,\bar a^5,\bar a^4\bar b,\bar a, \bar b)\quad\text{or}\quad\bar p=(\bar a,\bar a^{-1},\bar a^2\bar b,\bar b).
 \]

\begin{subcase}[B.1]$\bar{p} = (\bar{a}^3, \bar{a}\bar{b}, \bar{a}^5, \bar{a}^3\bar{b}, \bar{a}, \bar{a}^5\bar{b}).
$\par
Note that
\[
\bar{\pi}(\bar{a}^3) = \bar{\pi}(\bar{a}^5\bar{b}) = 1,\quad
\bar{\pi}(\bar{a}) = \bar{\pi}(\bar{a}\bar{b}) = 3 \quad
\text{and}\quad 
\bar{\pi}(\bar{a}^5) = \bar{\pi}(\bar{a}^3\bar{b}) = 5.
\]
As before, we see that $\Ker \varphi = \langle a^3, a^2b \rangle $ and $|X| = 6k$ for some $k\geq 1$. Since
the order of $\SD_{24}$ is twice the order of $\mathrm{D}_{12}$, we have $ k = 1$ or $2.$ 

If $k=1$, then we can set $ p = (c^3, x, c^{-1}, y, c, z) $, where
\[
\bar{c}^3 = \bar{a}^3,\quad 
\bar{x} = \bar{a}\bar{b},\quad 
\bar{y} = \bar{a}^3\bar{b}\quad 
\text{and}\quad 
\bar{z} = \bar{a}^5\bar{b}.
\]
It is easy to see that $|x| = |y| = |z| = 4, $ (cf. \eqref{Order}), so $X \neq X^{-1},$ a contradiction.

Thus, $k=2$, and we can set 
\[
p = (x_0, x_1, x_2, x_3, x_4, x_5, x_6, x_7, x_8, x_9, x_{10}, x_{11}) ,
\]  
 where $x_0 = a^3$ and
 \begin{align*}
 	\bar x_0 &= \bar x_6 = \bar{a}^3, &\bar x_1 &= \bar x_7 = \bar{a}\bar{b}, &\bar x_2& = \bar x_8 = \bar{a}^5, \\
 	\bar x_3 &= \bar x_9 = \bar{a}^3\bar{b},&\bar x_4 &= \bar x_{10} = \bar{a}, &\bar x_5& = \bar x_{11} = \bar{a}^5\bar{b}.
 \end{align*}
Since $\Core\varphi=\langle a^6\rangle=\{1,a^6\}$, we have $x_6=a^9$, and so
$x_0 = x_6^{-1}$. Similarly, $x_1,x_7\in\{ab,a^7b\}$ and $ x_1 = x_7^{-1};$ $x_3,x_9\in\{a^3b,a^9b\}$ and $x_3 = x_9^{-1}$;
$x_5,x_{11}\in\{a^5b,a^{11}b\}$ and  $ x_5 = x_{11}^{-1}.$ 
For the remaining elements, there are two possibilities:
either (a) $ x_2 = x_{10}^{-1}$ and $x_4 = x_8^{-1} $, or (b) $ x_2 = x_4^{-1} $ and $x_8 = x_{10}^{-1}.$

In Case (a),  using \eqref{Dist}, it is easy to compute:
\begin{align*}
\pi(x_0) &= \pi(x_5) = \pi(x_6) = \pi(x_{11}) = 1,\\
\pi(x_1) &= \pi(x_4) = \pi(x_7) = \pi(x_{10}) = 3,\\
\pi(x_2) &= \pi(x_3) = \pi(x_8) = \pi(x_9) = 11.
\end{align*}
If $ x_1 = ab $ and $ x_2 = a^5, $  then 
\[
x_4 = a ,\quad x_7 = a^7 b,\quad x_8 = a^{11}\quad \text{and}\quad x_{10} = a^7.
\]
It remains to determine $ x_3,x_5, x_9 $ and $ x_{11} $. We have
\[
a^5 = \varphi(ab) = \varphi(a)\varphi^3(b) = \varphi(a)\varphi(b) = \varphi(a) a^6b, 
\]
so $x_5=\varphi(x_4)=\varphi(a) = a^{11}b.$ Similarly, $x_9 = a^3b,$ $x_5=a$ and $x_{11}= a^5 b$. Therefore,
\[
p = (a^3, ab, a^5, a^9 b, a, a^{11} b, a^9, a^7 b, a^{11}, a^3 b, a^7, a^5 b).
\]
Similarly, if $ x_1 = ab $ and $ x_2 = a^{11} ,$ then $x_4 = a^7$, $x_7 = a^7 b$, $x_8 = a^5$ and $ x_{10} = a$, and so
\[
p = (a^3, ab, a^{11}, a^3 b, a^7, a^{11} b, a^9, a^7 b, a^5, a^9 b, a, a^5 b).
\]
If $ x_1 = a^7 b$ and $x_2 = a^5,$ then
\[
p = (a^3, a^7 b, a^5, a^3 b, a, a^5 b, a^9, ab, a^{11}, a^9 b, a^7, a^{11} b).
\]
If $ x_1 = a^7b$ and $x_2 = a^{11}$, then 
\[
p = (a^3, a^7 b, a^{11}, a^9 b, a^7, a^5 b, a^9, ab, a^5, a^3 b, a, a^{11} b).
\]
However, in each case $p$ cannot extend to a skew morphism, since
\[
\varphi(a^2b) = \varphi(a)\varphi^3(ab) \neq \varphi(a^{-1})\varphi^{11}(ab) = \varphi(a^2b).
\]

In Case (b),  we have
\[
\pi(x_0) = \pi(x_5) = \pi(x_6) = \pi(x_{11}) = 1,\]
\[\pi(x_1) = \pi(x_4) = \pi(x_7) = \pi(x_{10}) = 9,\]
\[
\pi(x_2) = \pi(x_3) = \pi(x_8) = \pi(x_9) = 5.
\]
As in the previous case, we get $p=p_1,p_2,p_3$ or $p_4$, where
\begin{align*}
p_1 &= (a^3, ab, a^5, a^3b, a^7, a^5b, a^9, a^7b, a^{11}, a^9b, a, a^{11}b),\\
p_2 &= (a^3, ab, a^{11}, a^9b, a, a^5b, a^9, a^7b, a^5, a^3b, a^7, a^{11}b),\\
p_3 &= (a^3, a^7b, a^5, a^9b, a^7, a^{11}b, a^9, ab, a^{11}, a^3b, a, a^5b),\\
p_4 &= (a^3, a^7b, a^{11}, a^3b, a, a^{11}b, a^9, ab, a^5, a^9b, a^7, a^5b).
\end{align*}
However ,  $p_2$ and $p_4$ cannot extend to a skew morphism, since
\[
\varphi(a^{10})=\varphi(a^9)\varphi(a)=a^8\neq \varphi(a)\varphi^9(a^9)=a^2.
\]
It is easy to see that if $\delta$ is the automorphism  of $\SD_{24}$ taking $a\mapsto a$ and $b\mapsto a^6b$,
then $ p_1^{\delta}=p_3 $,  so up to map isomorphism, we may assume $p=p_1$.
\end{subcase}

\begin{subcase}[B.2]$\bar{p} = (\bar{a}^3, \bar{a}^2\bar{b}, \bar{a}^5, \bar{a}^4\bar{b}, \bar{a}, \bar{b})$.\par
We may set $|X|=6k$ for $k=1,2$. Then
 $\bar {a}^3, \bar b \in \Ker \overline{\varphi}$, and hence $\Ker \varphi=\langle a^3,b\rangle$.

For $k=1$, we set $p=(x_0,x_1,x_2,x_3,x_4,x_5)$ where 
\[
\bar x_0 = \bar{a}^3,\quad 
\bar x_1 = \bar{a}^2\bar{b},\quad 
\bar x_2 = \bar{a}^5,\quad 
\bar x_3 = \bar {a}^4\bar{b}, \quad 
\bar x_4 = \bar{a},\quad
\text{and} \quad \bar x_5=\bar{b}.
\]
Since $\Core \varphi=\langle a^6 \rangle$, $|x_1|=|x_3|=|x_5|=2$ and  $|x_0|=|x_2|=|x_4|=4$, then $X \neq X^{-1}$, a contradiction.

For $k=2$,  we set 
\[
p = (x_0, x_1, x_2, x_3, x_4, x_5, x_6, x_7, x_8, x_9, x_{10}, x_{11}) ,
\]  
where $x_0 = a^3$ and
\begin{align*}
	\bar x_0 &= \bar x_6 = \bar{a}^3, &\bar x_1 &= \bar x_7 = \bar{a}^2\bar{b}, &\bar x_2& = \bar x_8 = \bar{a}^5, \\
	\bar x_3 &= \bar x_9 = \bar{a}^4\bar{b},&\bar x_4 &= \bar x_{10} = \bar{a}, &\bar x_5& = \bar x_{11} = \bar{b}.
\end{align*}
We have $|x_1|=|x_3|=|x_5|=|x_7|=|x_9|=|x_{11}|=2$ and $x_0=x_6^{-1}$, so $\chi(x_0)=6$. We calculate $\pi(x_0)=0-6+1=-5 \pmod{12}$. However, $x_0 \in \Ker \varphi$, a contradiction.
\end{subcase}

\begin{subcase}[B.3]$\bar p=(\bar a,\bar a^{-1},\bar a\bar b,\bar a^{-1}\bar b)$.\par
Note that $\bar a\bar b \in \Ker \bar{\varphi}$, so $\Ker \varphi=\langle a^3,ab\rangle$.
We have
\[
\overline{\varphi}=(\bar 1)(\bar b)(\bar a^3)(\bar a^3\bar b)(\bar a,\bar a^{-1},\bar a\bar b,\bar a^{-1}\bar b)(\bar a^4,\bar a^2,\bar a^4\bar b,\bar a^2\bar b).
\]
Thus, $\bar a^3$ is a fixed point of $\overline{\varphi}$, and so $\Core \varphi=\langle a^3\rangle$, a contradiction.
\end{subcase}
\begin{subcase}[B.4]$\bar p=(\bar a,\bar a^{-1},\bar a^2\bar b,\bar b)$.\par
We have
\[
\overline{\varphi}=(\bar 1)(\bar a\bar b)(\bar a^3)(\bar a^4\bar b)(\bar a,\bar a^{-1},\bar a^2\bar b,\bar b)(\bar a^4,\bar a^2,\bar a^5 \bar b,\bar a^3\bar b).
\]
Thus, $\bar a^3$ is a fixed point of $\overline{\varphi}$, and so $\Core \varphi=\langle a^3\rangle$, again a contradiction.
\end{subcase}
\end{case}
In summary, if $\Core\varphi = \langle a^3 \rangle,$ then $M\cong\CM(\SD_{24},X,p)$, 
where $p = (a, a^{-1}, a^4b, a^8b)$, while if $\Core\varphi = \langle a^6 \rangle,$
then $M\cong\CM(\SD_{24},X,p)$, where  
\[
p=(a^3, ab, a^5, a^3b, a^7, a^5b, a^9, a^7b, a^{11}, a^9b, a, a^{11}b).
\]

We have verified in Theorem~\ref{Class1} that these maps are indeed regular of skew-type $3$. 
Therefore, the proof is completed.
\end{proof}

%%%%%%%%%%
\begin{theorem}\label{Main2}
If $M=\CM(\SD_{8n}, X,p)$ is a regular Cayley map of skew-type $3$
on the semidihedral group $\SD_{8n}$ $(n>3)$, then it isomorphic to one of the maps $M(n,t)$
constructed in Theorem~\ref{Class2}.
\end{theorem}
%%%%%%
\begin{proof}
Assume  that $ M = \CM(\SD_{8n},X,p)  $ is a regular Cayley map of skew-type $3$
on the semidihedral group $\SD_{8n}$ for some $n>3$. 
Let $ \varphi $ be the corresponding skew morphism of $ M $ and $\pi$  the associated power function. 
 By Lemma \ref{lemma Core}, $n=3m$ $(m\geq 2)$ and 
 $\Core\varphi = \langle a^3 \rangle $, and by Lemma~\ref{quotient},  the quotient map 
 $\overline{M} = \CM(\overline{\SD}_{8n}, \bar{X}, \bar{p})$ is a
regular Cayley map of skew-type $3$ on the dihedral group $
\overline{\SD}_{8n} = \SD_{8n} / \Core\varphi \cong \mathrm{D}_6.$
By Proposition~\ref{D2n} and Lemma~\ref{auto}(a), we may set
\[
\bar{p} = (\bar{a}, \bar{a}^{-1}, \bar{a}\bar{b}, \bar{a}^{-1}\bar{b}).
\]
As before, we have $\Ker \varphi = \langle a^3, ab \rangle$ and $ |X| = 4k$ for some integer $k\geq1$. 
Using the notation in \eqref{Perm} and \eqref{Inverse},  we see that, for all $i\in\mathbb{Z}_k$,
\begin{align}
&\bar {x}_{4i} = \overline{a}, &&\bar {x}_{4i+1} = \overline{a}^{-1},&&\bar {x}_{4i+2} = \overline{a}\,\overline{b},&&
\bar {x}_{4i+3} = \overline{a}^{-1}\overline{b};\notag\\
&\bar\pi(\bar {x}_{4i}) \equiv 3, &&
\bar\pi(\bar {x}_{4i+1}) \equiv 2,&&
\bar\pi(\bar {x}_{4i+2}) \equiv 1,&&
\bar\pi(\bar {x}_{4i+3}) \equiv 2\pmod{4}.\label{QPower}
\end{align}
 We distinguish two cases.

%%%%%%%%%%%%%%%%%%%%
\begin{case}[A] $k=1$.\par
In this particular case, as in the proof of Theorem~\ref{Main1},  we may set $p = (a, a^{-1}, a^{3i+1}b, a^{3j-1}b)$, 
where $i,j$ are odd numbers with $3i-6mi\equiv 3j\pmod{12m}$. Since $i$ is odd, we may set $i=2s+1$, so that $a^{3i+1}b=a^{6s+4}b$ and $a^{3j-1}b=a^{6s+2+6m}b$.
Now the automorphism $\delta\in\SD_{8n}$ with $\delta(a)=a$ and $\delta(b)=a^{6s}b$ takes $(a,a^{-1},a^4b,a^{6m+2}b)$
to $(a,a^{-1},a^{6s+4}b,a^{6s+2+6m}b)$. Thus, we may assume $p=(a,a^{-1},a^4b,a^{6m+2}b)$ up to map isomorphism.

\end{case}

%%%%%%%%%%%
\begin{case}[B]$ k \geq 2 .$\par
As before, we see that $ab\in \Ker \varphi$, so $X_2\subseteq\Ker \varphi$, namely,  $\pi(x_{4i+2})=1$ for all $i$. 
By~\eqref{QPower}, we also see that, $\pi(x_{4i})\neq 1$ is odd, both $\pi(x_{4i+1})$ and $\pi(x_{4i+3})$ are even.
Since $\pi$ takes three distinct values in $\mathbb{Z}_{4k}$, we must have
\begin{align}
	\pi(x_0^{-1})=\pi(x_{4i+1})=\pi(x_{4i+3})\pmod{4k}.
\end{align}
Since $\pi(x_{4i+2})=1$, using \eqref{Dist} we get $\chi(x_{4i+2})=\chi(x_{4i+3}).$
 Since 
 \begin{align*}
\chi(x_{4i}) \equiv 1,\quad\chi(x_{4i+1}) \equiv -1,\quad \chi(x_{4i+2}) \equiv 0, \quad\chi(x_{4i+3}) \equiv 0 \pmod{4},
\end{align*}
we may assume 
\[
\chi(x_0) \equiv 4l_0+1 ,\quad \chi(x_{1}) \equiv 4l_1-1 ,\quad \chi(x_{4i+2})=\chi(x_{4i+3})\equiv 4l \pmod{4k},
\]
where $l_0,l_1,l\in\mathbb{Z}_k$. Then, again using \eqref{Dist}, we get
 \begin{align*}
 \pi(b)&=\pi(a)=\pi(x_0) \equiv  4l_1-4l_0-1 \pmod{4k},\\
  \pi(a^{-1})&=\pi(x_{1}) \equiv 4l-4l_1+2 \pmod{4k}.
  \end{align*}
Thus, from $\pi(a^{-1})=\pi(a^{-1}b)$ and using Lemma~\ref{Formula} we deduce that
\[
4l-4l_1+2=\pi(a^{-1}b) \equiv \sum_{i=0}^{\pi(a^{-1})-1} \pi(\varphi^i(b)) \equiv (4l_1-4l_0-1)(4l-4l_1+2) \pmod{4k},
\]
which reduces to $4(l_1-l_0)-2(l-l_1)+2(l_1-l_0)-1 \equiv 0 \pmod{k}.$ Thus, $k$ is odd.

 By Lemma~\ref{CD}, we may set $\chi(x_{i+4})-\chi(x_i) \equiv d \pmod{4k}$ for all $i$.
Since both $x_{4i+2}$ and $x_{4i+3}$ are elements of $b$-type, by \eqref{Order}, 
 $|x_{4i+2}| = |x_{4i+3}| = 2$ or $4$ for any $i$. We distinguish three subcases.
 
 %%%%%%%%%%%%%%%%
\begin{subcase}[B.1]  $|x_{4i+2}| = |x_{4i+3}| = 4$ for all $i$.\par
In this subcase, by Lemma~\ref{Lift}, the subset $X_2:=\{x_i\mid i\equiv 2\pmod{4}\}$ is closed under inverses 
(and contains no involutions), so $k=|X_2|$ is even, a contradiction. 
\end{subcase}
 
 %%%%%%%%%%%%
 \begin{subcase}[B.2]$|x_{4i+2}| = |x_{4i+3}| = 2$ for all $i$.\par
By evaluating $\pi(x_0^{-1})$ in two ways:
\begin{align*}
\pi(x_0^{-1}) &= \pi(x_{4i-1}) 
\equiv \chi(x_{4i}) - \chi(x_{4i-1}) + 1 
= \chi(x_{4i}) + 1,\\
\pi(x_0^{-1}) &= \pi(x_{4i+1}) 
= \chi(x_{4i+2}) - \chi(x_{4i+1}) + 1 
=-\chi(x_{4i+1}) + 1,
\end{align*}
we get
\begin{equation}\label{chivalue}
\chi(x_{4i}) + \chi(x_{4i+1}) \equiv 0 \pmod{4k}.
\end{equation}
Thus, we may assume that
\begin{align}\label{Form4}
	 \chi(x_0) = 4l + 1 \quad \text{and} \quad \chi(x_{1})  = 4k - 4l - 1,
\end{align}
where  $ l \in\mathbb{Z}_k,$ and so  $\pi(x_0) = -8l - 1 $ and 
$ \pi(x_1) =  4l + 2.$
Recall that $\pi(x_2)=1$ and $\pi(x_3)=\pi(x_1)$. Since $\pi$ has period $4$, we get
\[
\pi(x_i)=
\begin{cases}
-8l-1, & i\equiv0\pmod{4},\\
4l+2, &  i\equiv1\pmod{4},\\
1, &  i\equiv2\pmod{4},\\
4l+2, &  i\equiv1\pmod{4}.
\end{cases}
\]
Therefore, for all $i\in\mathbb{Z}_k$, we have 
\[
\sum_{j=0}^3\pi(x_{i+j})=(-8l-1)+(4l+2)+1+(4l+2)\equiv 4\pmod{4k}.
\]
Since $\SD_{8n}=\langle X\rangle$, we deduce that $\varphi^4$ is an automorphism of $\SD_{8n}$.

Moreover, since $\bar b$ is a fixed point of $\overline{\varphi}$, the orbit $O_b$ is smooth 
so that $\pi(\varphi^i(b)) = \pi(b) \equiv 4l + 2$ for any integer $i$. It follows that
\begin{align*}
4l+2&=\pi(a^{-1}b) \equiv (4l + 2)\big(2(-4l - 1) + 1\big)\pmod{4k},\\
 1 &= \pi(b^2) \equiv \big(2(-4l - 1) + 1\big)^2  \pmod{4k}.
\end{align*}
or equivalently, $8l^2 + 6l + 1 \equiv 0 \pmod{k}$ and $16l^2 + 4l \equiv 0 \pmod{k}.$ 
Since $k$ is odd, the second congruence is reduced to
 $ 4l^2 + l \equiv 0 \pmod{k}.$  Therefore,  
 \[
 4l + 1 \equiv 8l^2 + 6l + 1 - 2(4l^2 + l) \equiv 0 \pmod{k},
 \]
 and hence, we deduce from \eqref{Form4} that $\chi(x_0)=\chi(x_1)=0\pmod{k}$. By Lemma~\ref{CD}, the 
 common difference of $\chi$ is $d=0$, thus, $ \chi(x_{4i})\equiv\chi(x_{4i+1})=0\pmod{k}$
for all $i\in\mathbb{Z}_k$. 
Consequently, $\chi(x_i)\equiv0\pmod{k}$ for any $i\in\mathbb{Z}_k$. It follows that 
\[
\sum_{j=0}^{k-1}\pi(x_{i+j})=\sum_{j=0}^{k-1}(\chi(x_{i+j+1})-\chi(x_{i+j})+1)\equiv 0\pmod{k}
\]
for all $x_i\in X$. Therefore, $ \varphi^k $ is a skew morphism on $ \SD_{8n} $. Recall that $ k $ is odd, 
 $k \equiv 1$ or $ -1 \pmod{4} .$ If we set 
\[
\alpha = 
\begin{cases} 
	\varphi^k, & k \equiv 1 \pmod{4}, \\
	\varphi^{3k}, & k \equiv -1 \pmod{4}, 
\end{cases} 
\quad \text{and} \quad 
\beta = 
\begin{cases} 
\varphi^{3k+1}, & k \equiv 1 \pmod{4}, \\
\varphi^{k+1}, & k \equiv -1 \pmod{4}.
\end{cases}
\]
It follows that $\beta$ is an automorphism of $\SD_{8n}$ and $\varphi = \alpha\beta$. Note that $\CM(\SD_{8n}, Y, \alpha|_Y) $
is a regular Cayley map of skew-type 3 and valence 4, where  $ Y $ is the $ \alpha $-orbit of $ a $.
By the result obtained in Case (A), up to map isomorphism, we may set
\begin{equation}\label{alpha}
\alpha(a^i b^j) = 
\begin{cases} 
	a^{(2m-1)i+6j}b^j, & i \equiv 0 \pmod{3},\\
	a^{(2m-1)i-2m+6j}b^j, & i \equiv 1 \pmod{3},\\
	a^{(2m-1)i+2m+3}b^{j+1}, & i \equiv -1 \pmod{3}.
	\end{cases}
\end{equation}
On the other hand,  we may assume $\beta\in\Aut(\SD_{8n}))$ has the form
\[
\beta(a) = a^t\quad\text{and}\quad \beta(b) = a^s b,
\] 
where $t, s \in\mathbb{Z}_{4n},$ $\gcd(t, 4n) = 1$, and $s$ is even. Since $|\beta|=k$, we have
\[
a=\beta^k(a)=a^{t^k}\quad\text{and}\quad b=\beta^k(b)=a^{s\sum_{i=1}^kt^{i-1}}b.
\]
Thus, $t^k \equiv 1 \pmod{4n}$. Recall that $3|n$ and $k$ is odd, it follows that $t \equiv 1 \pmod{3}$.
On the other hand,  since $\bar b=\overline{\beta(b)}= \bar a^s\bar b$, we have $s \equiv 0 \pmod{3}.$ 
Now we use $ \alpha \beta = \beta \alpha $   to compute:
\begin{align*}
&a^{6t+s}b =\beta (a^{6}b)\stackrel{\eqref{alpha}}= \beta \alpha(b)=\alpha \beta(b)  =\alpha (a^s b)\stackrel{\eqref{alpha}}= a^{(2m-1)s+6}b, \\
&a^{5t+s}b =\beta(a^5b) =\beta \alpha(ab) = \alpha \beta(ab) = \alpha(a^{t+s}b) = a^{(2m-1)(t+s)-2m+6}b, \\
& a^{4t}= \beta (a^4) =\beta \alpha(a^{-1}b) =\alpha \beta(a^{-1}b) = \alpha (a^{-t+s}b) = a^{(2m-1)(s-t)+2m+3}. \\
\end{align*}	 
Thus,
\begin{align*}
&(2m-2)s-6t+6\equiv0\pmod{12m},\\
&(2m-2)s+(2m-6)t-2m+6\equiv0\pmod{12m},\\
&(2m-1)s-(2m+3)t+2m+3\equiv0\pmod{12m}.
\end{align*}
The first two congruences give $2mt\equiv 2m\pmod{12m}$, or equivalently, $t \equiv 1 \pmod{6}$.
Upon substitution, the third one reduces to $(2m-1)s\equiv 3(t-1)\pmod{12m}$, which together with 
the first gives $s \equiv 3 - 3t \pmod{12m} $. Therefore,  for all $ a^i b^j \in \SD_{8n} $,  
\[
\varphi(a^i b^j)  = \alpha (a^{it+(3-3t)j} b^j) \stackrel{\eqref{alpha}}=
\begin{cases} 
	a^{(2m-1)t i - 3(2m-1)t j + 3(2m-1)j + 6j} b^j, & i \equiv 0 \pmod{3}, \\
	a^{(2m-1)t i - 3(2m-1)t j + 3(2m-1)j - 2m + 6j} b^j, & i \equiv 1 \pmod{3}, \\
	a^{(2m-1)t i - 3(2m-1)t j + 3(2m-1)j + 2m + 3} b^{j+1}, & i \equiv -1 \pmod{3}.
\end{cases}
\]
Recall that $\bar a=\bar c$, so $c=a^{3e+1}$ for some integer $e$. 
Therefore, $\SD_{8n}=\langle X \rangle = \langle c,b\rangle$, so $ c $ is a generator of $ \langle a \rangle $. We
may choose $\delta \in \Aut(\SD_{8n}) $ with $\delta(a) = c$ and $\delta(b) = b$, so that $p$
has the the stated form (see~\eqref{Perm0}). Consequently, $M \cong M(n,t).$
 \end{subcase}

%%%%%%%%%%%
\begin{subcase}[B.3]$|x_{4i+2}| = |x_{4i+3}| =2 $ and $|x_{4j+2}| = |x_{4j+3}| = 4,$ for some distinct $i,j$.\par
In this subcase,  we remark that there exist no consecutive numbers $i$ and $i+1$ with 
$|x_{4i+2}| = |x_{4i+3}| = 2$ and $|x_{4(i+1)+2}| = |x_{4(i+1)+3}| = 2$, otherwise, 
we would have the common difference $d=0$, so $|x_{4i+2}| = |x_{4i+3}| = 2$ for all $i$, a contradiction.

Without loss of generality, we assume $|x_2|=|x_3|=2$ and $|x_6|=|x_7|=|x_{4k-2}|=|x_{4k-1}|=4.$ Then $\chi(x_2)=\chi(x_3)=0$.
Since
\begin{align*}
	\chi(x_{4i}) \equiv 1 \pmod{4}\quad\text{and}\quad\chi(x_{4i+1}) \equiv -1 \pmod{4},
\end{align*}
we may set set $\chi(x_0)=4l_0+1$ and $\chi(x_1)=4l_1-1$, where $l_0,l_1\in\mathbb{Z}_k$. By Lemma~\ref{CD}, we have
\[
\chi(x_{4i})=di+4l_0+1,\quad\chi(x_{4i+1})=di+4l_1-1,\quad \chi(x_{4i+2})=\chi(x_{4i+3})=di\pmod{4k},
\]
for all $i\in\mathbb{Z}_k$. It follows that
\[
2-4l_1\stackrel{\eqref{Dist}}=\pi(x_1)=\pi(x_3)\stackrel{\eqref{Dist}}=d+2+4l_0\pmod{4k},
\]
so $d+4(l_0+l_1)\equiv0\pmod{4k}$.  Since $\varphi$ has period $4$, by Lemma~\ref{CD}, $d \equiv 0\pmod{4}.$ 
Now rewrite  $k=2s+1$ (recall that $k$ is odd), then
\begin{align*}
\chi(x_{4(s+1)})+\chi(x_{4(s+1)+1})&=(d(s+1)+4l_0+1)+(d(s+1)+4l_1-1)\\
&=d(2s+1)+d+4(l_0+l_1)\equiv0\pmod{4k},
\end{align*}
so $x_{4(s+1)}^{-1}=x_{4(s+1)+1}$. Thus, $\chi(x_{4(s+1)})=1$ and $\chi(x_{4(s+1)+1})=-1\pmod{4k}$. Therefore,
$\pi(b)\equiv\pi(a)\equiv\pi(x_0)=\pi(x_{4(s+1)})=-1$ and 
\[
\pi(a^{-1})\equiv\pi(x_1)=\pi(x_{4(s+1)+1})=\chi(x_{4(s+1)+2})-\chi(x_{4(s+1)+1})+1=d(s+1)+2
\]
Now
\[
d(s+1)+2=\pi(a^{-1})=\pi(a^{-1}b)=\pi(a^{-1})\pi(b)=-(d(s+1)+2)\pmod{4k},
\]
which reduces to $2d(s+1)+4\equiv0\pmod{4k}$, or equivalently, $d(2s+1)+d+4\equiv0\pmod{4k}$. Since $k=2s+1$, 
this is equivalent to $d\equiv -4\pmod{4k}$. Thus, $\chi(x_6)=d=-4$, so $x_6^{-1}=x_2$. However, $x_2$ is an involution and 
$x_6$ is not, a contradiction.
\end{subcase}
\end{case}
\end{proof}

\end{document}